\documentclass[leqno,11.5pt]{article}
\usepackage{latexsym}
\usepackage{amsmath}
\usepackage{amssymb}
\usepackage{enumerate}
\usepackage{graphicx}
\usepackage[all]{xy}

\usepackage{theorem}
\newtheorem{theorem}{Theorem}[section]
\newtheorem{proposition}[theorem]{Proposition}
\newtheorem{lemma}[theorem]{Lemma}
\newtheorem{corollary}[theorem]{Corollary}

\theorembodyfont{\rmfamily}
\newtheorem{proof}{\textmd{\textit{Proof.}}}

\newtheorem{remark}[theorem]{Remark}

\newtheorem{definition}[theorem]{Definition}

\newtheorem{acknowledgement}{\textmd{\textit{Acknowledgements.}}}

\makeatletter

\@addtoreset{equation}{section}
\makeatother

\newcommand{\qedd}{\hfill \Box}
\newcommand{\ve}{\varepsilon}

\newcommand{\ora}{\overrightarrow}

\newcommand{\wt}{\widetilde}
\newcommand{\wh}{\widehat}
\newcommand{\ol}{\overline}
\newcommand{\divi}{\divideontimes}
\newcommand{\cdas}{\circledast}
\newcommand{\B}{\ensuremath{\mathbb{B}}}
\newcommand{\N}{\ensuremath{\mathbb{N}}}
\newcommand{\R}{\ensuremath{\mathbb{R}}}

\newcommand{\Sph}{\ensuremath{\mathbb{S}}}

\newcommand{\cF}{\ensuremath{\mathcal{F}}}

\newcommand{\cL}{\ensuremath{\mathcal{L}}}
\newcommand{\cM}{\ensuremath{\mathcal{M}}}

\newcommand{\cU}{\ensuremath{\mathcal{U}}}

\def\adj{\mathop{\mathrm{adj}}\nolimits}
\def\Bd{\mathop{\mathrm{Bd}}\nolimits}
\def\inj{\mathop{\mathrm{inj}}\nolimits}

\def\ra{\mathop{\mathrm{rank}}\nolimits}

\def\grad{\mathop{\mathrm{grad}}\nolimits}
\def\supp{\mathop{\mathrm{supp}}\nolimits}
\def\id{\mathop{\mathrm{id}}\nolimits}

\def\Cut{\mathop{\mathrm{Cut}}\nolimits}

\def\L{\mathop{\mathrm{Lip}}\nolimits}

\def\Conv{\mathop{\mathrm{Conv}}\nolimits}
\def\Im{\mathop{\mathrm{Im}}\nolimits}
\def\Ker{\mathop{\mathrm{Ker}}\nolimits}

\def\Sing{\mathop{\mathrm{Sing}}\nolimits}

\title{
\Large{
Approximations of Lipschitz maps 
via Ehresmann fibrations and Reeb's 
sphere theorem 
for Lipschitz functions}
\\
{\small {\it To the memory of Kazumasa Kuwada}}
}

\author{Kei KONDO\footnote{
Supported by 
the Grant-in-Aid for Science Reserch (C), 
JSPS KAKENHI Grant Numbers 
16K05133, 17K05220, 18K03280.}
}
\date{\today}
\begin{document}
\maketitle

\begin{abstract}
We show, as our main theorem, that if 
a Lipschitz map from a compact Riemannian 
manifold $M$ to a connected compact 
Riemannian manifold $N$, 
where $\dim M \ge \dim N$, has no singular 
points on $M$ in the sense of F.H.\,Clarke, 
then the map admits a smooth approximation via 
Ehresmann fibrations. We also show the Reeb sphere theorem 
for Lipschitz functions, i.e., if a closed Riemannian 
manifold admits a Lipschitz function with exactly two singular 
points in the sense of Clarke, then the manifold 
is homeomorphic to the sphere. 
\end{abstract}


\section{Introduction}\label{sec1}

\subsection{Background: Grove--Shiohama theory 
for distance functions}

Armed with the Toponogov Comparison Theorem 
\cite{Toponogov} (see also \cite{CE}, \cite{Sak}), K. Grove and K. Shiohama \cite{GS} developed a theory for critical points of distance functions on complete Riemannian manifolds that has played a fundamental role in the study of relationships between curvature and topology. Denote by $X$ a complete Riemannian manifold, 
$d$ its distance function, and $T_xX$ the tangent space at each $x \in X$. Fix $p \in X$, and set $d_p (x):= d(p, x)$ for all $x \in X$. 
Note that $d_p$ is a $1$-Lipschitz function and is smooth 
on $X \setminus (\{p\} \cup \Cut (p))$ where $\Cut (p)$ 
indicates the cut locus of $p$. (For basic definitions in Riemannian geometry see, for example, \cite{CE}, \cite{dC}, \cite{Sak}.) 
Grove and Shiohama gave the following 
meaningful definition in order to do research into 
how $d_p$ behaves. 

\begin{definition}{\rm (\cite{GS})}\label{2018_07_01_def_GS}
A point $q \in X\setminus \{p\}$ is said to be 
{\em critical for $d_p$} (or a {\em critical point of $d_p$}) 
{\em in the sense of Grove--Shiohama} 
if for each $v \in T_q X \setminus \{o_q\}$ 
there is a unit speed minimal geodesic segment 
$\gamma :[0, d_p(q)] \to X$ emanating from 
$p = \gamma (0)$ to $q = \gamma (d_p(q))$ 
such that 
$\angle 
(
- (d\gamma / dt )(d_p(q)), v  
)
\le \pi / 2$ 
where $\angle 
(- (d\gamma / dt)(d_p(q)), v)$ 
denotes the angle between two vectors 
$- (d\gamma / dt)(d_p(q))$ and $v$ in $T_qX$. 
For convenience we also call 
$p$ a critical point of $d_p$.
\end{definition}
The origins of this definition can be found in the work 
of M. Berger \cite{Ber}: the point of maximal distance from a given point $x \in X$ is a critical point of $d_x$. See the survey articles by 
J. Cheeger \cite{Ch} and by K. Grove \cite{Grove} 
on critical points of distance functions. 
Note that any critical point of $d_x$ is also a cut point of $x$.\par 
Another major development, due to M. Gromov \cite{G1}, 
was in the topology of regions free of critical points.

\begin{lemma}{\bf (Gromov's isotopy lemma)}\label{2018_11_03_Gromov}
If $0 < R_1 < R_2 \le \infty$, 
and if $d_p$ has no critical points on 
$\overline{B_{R_2}(p)} \setminus B_{R_1}(p)$, 
then $\overline{B_{R_2}(p)} \setminus B_{R_1}(p)$ is homeomorphic to 
$\partial B_{R_1}(p) \times [R_1, R_2 ]$ 
where each $B_{R_{i}}(p)$ denotes 
the metric open ball with center $p$ and radius $R_i$, and 
$\overline{B_{R_{i}}(p)}$ indicates the closure of 
$B_{R_{i}}(p)$ ($i=1,2$).
\end{lemma}

The Toponogov comparison 
theorem \cite{Toponogov} (see also \cite{CE}, \cite{Sak}) 
together with the isotopy lemma yields 
the {\em diameter sphere theorem}:  

\begin{theorem}{\rm (\cite{GS})}\label{2018_11_07_GS}
If the sectional curvature of $X$ is bounded from below 
by $1$, and if the diameter of $X$ is greater 
than $\pi/2$, then $X$ is homeomorphic to the sphere.   
\end{theorem}

\subsection{Critical points of Lipschitz functions}
The method of Grove and Shiohama has many applications (see \cite{AbGr}, \cite{G1}, \cite{GP}, \cite{KO}, \cite{KT2010}, 
and the survey articles \cite{Ch}, \cite{Grove}). 
A natural question to ask is whether it can be extended to general Lipshitz functions.\par 
The purpose of this article is to tackle this question by employing 
Clarke's nonsmooth analysis. 
That is, we will extend the notion of critical points of distance functions on Riemannian manifolds to locally Lipschitz maps. In the absence of singular points we will show the existence of a family of Ehresmann fibrations 
which approximate an arbitrary Lipschitz map between 
compact manifolds without curvature assumption 
(Theorem \ref{2018_01_04_maintheorem}). Moreover we will 
show the Reeb sphere theorem for Lipschitz functions on closed Riemannian manifolds (Theorem \ref{2019_06_08_thm1.7}) which corresponds 
to that for smooth ones \cite{Ree}, \cite{Milnor2}.

\subsection{Main theorem}
Let $M$ and $N$ be smooth manifolds. 
A smooth map $f:M \to N$ is called 
{\em an Ehresmann fibration} (or a {\em locally trivial fibration}) 
if for each $x\in N$ there are an open neighborhood 
$U_x$ of $x$ and a diffeomorphism 
$g: f^{-1}(U_x) \to U_x\times f^{-1}(x)$ such that the diagram 
\[
\xymatrix{
f^{-1}(U_x) \ar[rr]^-g \ar[rd]_{f|_{f^{-1}(U_x)}} & & \ U_x\times f^{-1}(x) \ar[ld]^{\pi} \\
& U_x & 
}
\]
commutes where $\pi: U_x\times f^{-1}(x) \to U_x$, 
$\pi (p,q):=p$, 
denotes the projection to the first factor. Note that 
$\pi$ is a smooth map. 
Our main theorem is stated as follows: 

\begin{theorem}\label{2018_01_04_maintheorem}
{\rm (Main Theorem)} 
Let $F: M \to N$ be a Lipschitz map from a compact Riemannian manifold 
$M$ to a connected compact Riemannian manifold $N$ 
where $\dim M \ge\dim N$. 
If $F$ has no singular points on $M$ 
in the sense of Clarke, then 
for any $\eta>0$ there is a constant 
$\kappa(\eta) >0$ such that for each $\ve \in (0, \kappa (\eta))$ 
there is an Ehresmann fibration
$f_\ve$ from $M$ onto $N$ satisfying 
$\max_{x\in M}d_N(f_\ve(x), F(x))<\eta$.
\end{theorem}

\begin{remark}
Let us mention remarks on 
Theorem \ref{2018_01_04_maintheorem}: 
\begin{enumerate}[{\rm (i)}]
\item The definition of a singular point of Lipschitz 
maps in the sense of Clarke will be given 
in Section \ref{2018_sec2.1}.
\item The author and Tanaka showed 
the existence of a family of immersions which approximate an arbitrary Lipschitz map between compact manifolds, see 
\cite[Theorem 1.3]{KT}. Li \cite{Li} announced 
an another proof of \cite[Theorem 1.3]{KT}. 
\cite[Corollary 1.15]{KT} guarantees that an assumption 
on \cite[Proposition 22]{GH} is natural.
\item The related result is the Yamaguchi fibration 
theorem \cite{Yam1}: Let $X$ and $Y$ be 
complete Riemannian manifolds of $\dim X =n$ and 
$\dim Y =k$, respectively, where $n\ge k$. 
Assume that both sectional curvatures are 
bounded from below by $-1$, and that the injectivity radius 
of $Y$ has a lower bound $\delta >0$. 
He then showed that there is a constant 
$\ve (n, \delta) >0$ such that if 
$d_{{\rm GH}} (X, Y) < \ve (n, \delta)$, then there is a fibration $f:X\to Y$ which is an almost Riemannian submersion 
where $d_{{\rm GH}}$ indicates the Gromov--Hausdorff distance. He also gave this type of fibration 
theorem for Alexandrov spaces \cite{Yam2}. 
Moreover Fujioka \cite{Fuji} showed 
a locally trivial fibration theorem for 
Alexandrov spaces assuming 
a lower positive bound for the volume of the space of directions.
\end{enumerate}
\end{remark}

\subsection{Reeb's sphere theorem for Lipschitz functions}

In the process of proving Theorem \ref{2018_01_04_maintheorem} we obtain 
the following corollary of a proposition for Lipschitz maps between Riemannian manifolds: 

\begin{corollary}{\rm 
(Corollary \ref{2018_09_21_cor5.8} 
in Section \ref{2018_02_16_Sect5})}\label{2018_11_07_cor1.6}
Let $F$ be a Lipschitz function on a compact 
Riemannian manifold $M$, and 
$\wt{F}_\ve: M\to \R$ the global smooth approximation of $F$ (see Definition \ref{2018_02_25_def4.9} for $\ell =1$). 
If $p \in M$ is nonsingular for $F$ in the sense of 
Clarke, then there are two constants 
$\lambda (p) >0$ and $\ve_0 (p) >0$ such that 
if $\ve \in (0, \ve_0 (p))$, then 
$\grad \wt{F}_\ve \not =0$ on the metric open ball 
$B_{\lambda (p)}(p)$ with center $p$ 
and radius $\lambda (p)$ 
where $\grad \wt{F}_\ve$ denotes the gradient vector field of $\wt{F}_\ve$. 
In particular $\wt{F}_\ve$ has no critical points 
on $B_{\lambda (p)}(p)$ for an $\ve >0$ 
sufficiently small. 
\end{corollary}

Applying this corollary we show Reeb's sphere theorem for 
Lipschitz functions: 

\begin{theorem}\label{2019_06_08_thm1.7} 
If a closed Riemannian 
manifold admits a Lipschitz function with exactly 
two singular points in the sense of Clarke, then the manifold 
is homeomorphic to the sphere.
\end{theorem}

\begin{remark}\label{2018_08_09_rem1.6} 
We give two remarks on Theorem \ref{2019_06_08_thm1.7}: 
\begin{enumerate}[{\rm (i)}]
\item Let $X$ be a closed Riemannian manifold, $p \in X$, and $d_p$ the distance function of $X$ given by $d_p (x):= d(p, x)$ for all $x \in X$. We then see 
that a point $x \in X$ is critical 
for $d_p$ in the sense of 
Grove--Shiohama if and only if $x \in X$ is singular 
for it in that of Clarke (Proposition \ref{2018_07_09_lem2.27} and Lemma \ref{2018_07_11_lem2.28}). 
Theorem \ref{2019_06_08_thm1.7} thus 
contains Reeb's sphere theorem for distance functions \cite[Proposition 2.10]{Sak}, and hence 
Theorem \ref{2019_06_08_thm1.7} 
yields Theorem \ref{2018_11_07_GS}.
\item Since the manifold in Theorem \ref{2019_06_08_thm1.7}, 
denoted by $M$ below, is a twisted sphere, we see, by \cite{Cerf}, \cite{KervaireMilnor}, \cite{Milnor3}, \cite{Palais}, 
and \cite{Sm0}, that $M$ is diffeomorphic to 
the standard sphere when $\dim M \le 6$. Moreover 
the Weinstein deformation technique for metrics (\cite{W})
shows that $M$ admits a metric such that there is a point 
whose cut locus consists of a single point. It is the worthy of noting that every exotic sphere of 
dimension greater than $4$ admits such metrics 
by the Smale $h$-cobrdism theorem \cite{Sm1}, \cite{Sm2} together with the deformation technique.
\end{enumerate}
\end{remark}

The article is organized as follows. 
In Section \ref{2016_01_24_sec2} we define 
a generalized differential for Lipschitz 
maps between Riemannian manifolds (Definitions \ref{2018_02_01_def2.3} and \ref{2018_01_09_def2.2}) 
and singular points 
of them in the sense of Clarke 
(Definition \ref{2018_01_10_def2.7}). 
Giving intrinsic definitions to them is 
another aim of this article. 
That is, although we had given the definitions in \cite{KT}, the identification, which has often been 
done in \cite{KT}, 
of the set of all linear mappings of tangent spaces 
and the vector space of matrices seems to 
have given not only an impression that it is hard 
to read it, but also a misunderstanding that 
the definitions depend on the choice of charts. 
To prevent them we thus employ parallel transports along minimal geodesics in our definitions. This is 
the big difference between our definitons 
and those in \cite{KT}. 
Moreover we also define 
the generalized gradient for Lipschitz functions on 
Riemannian manifolds (Definition \ref{2018_07_08_def2.18}) and study the relationship 
between the gradient and the generalized differential of 
them. As an example of singular points of Lipschitz functions we show finally that critical points of distance functions in the sense of Grove--Shiohama are singular ones of them in that of Clarke 
(Proposition \ref{2018_07_09_lem2.27} and 
Lemma \ref{2018_07_11_lem2.28}).\par 
In Section \ref{2018_01_18_Sect3} 
we define the adjoint of the generalized differential 
of Lipschitz 
maps between Riemannian manifolds (Definition \ref{2018_11_08_def3.8}), and discuss 
surjectivity and injectivity of the generalized differential 
near a nonsingular point of a given Lipschitz 
map (Propositions \ref{2018_02_12_prop3.11} and \ref{2018_01_28_prop2.12}, respectively). 
These propositions show that 
the set of all singular points of the map 
is a closed set in its  source space 
(Corollary \ref{2018_09_20_cor3.14}).\par 
In Section \ref{2018_01_18_Sect4} 
we first define a local smooth approximation of an arbitrary Lipschitz map between Riemannian 
manifolds on a strongly convex ball as the Riemannian convolution smoothing 
(Definiton \ref{2018_01_20_def4.1}), and 
next define the global smooth approximation of the map via a smooth partition of unity 
(Definiton \ref{2018_02_25_def4.9}).\par 
In Section \ref{2018_02_16_Sect5} we give the proof 
of Theorem \ref{2018_01_04_maintheorem}. For this we first show, broadly speaking, that a global 
smooth approximation of a Lipschitz map $F$ on a compact manifold carries on surjectivity 
of the generalized differential of $F$ (Proposition \ref{2018_03_06_lem5.6}). 
As a corollary of Proposition \ref{2018_03_06_lem5.6} we get 
Corollary \ref{2018_11_07_cor1.6}. 
Using the proposition and 
the tubular neighborhood theorem, 
we finally show the main theorem.\par 
In Section \ref{2019_06_08_Sect6}, 
making use of Corollary \ref{2018_11_07_cor1.6} with 
Morse theory, we show 
Theorem \ref{2019_06_08_thm1.7}.

\begin{acknowledgement}
The author expresses his sincere thanks to Professors 
M. Koiso, A. Katsuda, A. Mitsuishi, and Y. Otsu 
who attended his intensive course on \cite{KT}, 
organized by Koiso, at Kyushu university in the summer of 2017 and their many critical remarks, which started him considering 
intrinsic definitions of the generalized differential 
and singular points of Lipschitz maps. 
He wishes to thank Professor N. Nakauchi for helpful discussions about the intrinsic definitions in 
Section \ref{2016_01_24_sec2}. 
He wishes to express his appreciation 
to the referee for his cooperation and helpful comments. 
In particular he has been able to improve Theorem \ref{2019_06_08_thm1.7} due to the 
comments. And last but not least he is grateful to Professor K. Shiohama for his strong encouragement during the preparation of this work. 
\end{acknowledgement}

\section{Nonsmooth analysis in Riemannian geometry}\label{2016_01_24_sec2}

\subsection{From Rockafellar to Clarke}\label{2020_09_12_sub2.1}
Rockafellar \cite{Ro} was the first to introduce the notion of the subdifferential of a convex function. This was done on of Euclidean space in order to replace assumptions of smoothness with convexity and led to many results. Clarke \cite{C1}, \cite{C2} generalized Rockafellar's work to Lipschitz maps between Euclidean spaces and the subdifferential to the generalized 
differential (see Definition \ref{2018_01_09_def2.2}).\par 
The two examples below show how Clarke's generalized differential of Lipschitz maps emerges from Rockafellar's subdifferential of convex functions.
\begin{enumerate}[{\rm (i)}] 
\item 
We here recall Rockafellar's ideas, that is the 
subdifferential of convex functions. 
Let $f_1 (x):= |x|-1$ and $f_2 (x):= (x-2)^2-1$ for 
all $x \in \R$. 
Define the convex function $f:\R\to \R$ 
by $f(x):= \max\{f_1 (x), f_2(x)\}$ ($x \in \R$). 
Note that $f$ is not differentiable on $\{1, 4\}$. 
However one-sided limits of $f'$ do exist, i.e., 
$\lim_{x\uparrow 1}f'(x)=-2$, 
$\lim_{x\downarrow 1}f'(x)=1$, 
$\lim_{x\uparrow 4}f'(x)=1$, 
and $\lim_{x\downarrow 4}f'(x)=4$. 
Rockafellar's idea is 
to draw vertical segments between disconnected points 
of the graph of $f'$ using convex combinations: 
between $(1, -2)$ and $(1, 1)$ and between $(4,1)$ and $(4,4)$. 
Put differently, for each $\lambda \in [0,1]$ we have 
$(1-\lambda) \lim_{x\uparrow 1}f'(x)+ \lambda \lim_{x\downarrow 1}f'(x)= 3\lambda -2$ and 
$(1-\lambda) \lim_{x\uparrow 4}f'(x)+ \lambda \lim_{x\downarrow 4}f'(x)
= 3\lambda +1$, 
and hence 
$
\partial f(1):=\{3\lambda -2\,|\,\lambda \in [0,1]\}
= [-2,1]$ and 
$\partial f(4):=\{3\lambda +1\,|\,\lambda \in [0,1]\}
= [1,4]$. 
Since $0\in \partial f(1)$, and since $f$ is not 
monotone near $x=1$, we can regard $x=1$ as a 
critical point of $f$. In particular $f$ has the minimum value 
$0$ at $x = 1$. On the other hand 
we can regard $x=4$ as a noncritical point of $f$, 
as $0 \not\in \partial f(4)$, and 
$f$ is increasing near $x=4$. He called 
$\partial f(1)$ and $\partial f(4)$ the {\em subdifferentials} of 
$f$ at $x=1,4$, respectively.
\item 
Let $g :\R\to\R$ be 
the locally Lipschitz function defined by 
\[
g(x) = 
\begin{cases}
\ x^2 \sin \displaystyle{\frac{1}{x}} \ & (x\not=0),\\
\ 0  \ & (x=0).
\end{cases}
\]
Note that $g$ is differentiable on $\R$, 
but is not $C^1$ at $x=0$. 
Moreover we can not directly apply 
Rockafellar's idea as in example (i) 
to one-sided limits of $g'$ at $x = 0$ due to 
the term $\cos (1/x)$ in $g' (x)$. 
Clarke's idea is to choose a sequence of lines with the same slope tangent to 
the graph of $g$, or, more precisely, 
for each $\alpha \in [-1,1]$ we choose 
a sequence $\{x_i^{(\alpha)}\}_{i\in \N} \subset \R$ 
which converges to $0$ 
as $i\to\infty$ such that 
$\lim_{i\to\infty} g'(x_i^{(\alpha)})=\alpha$, and 
take the convex hull, denoted by $\Conv (A)$, of 
the (nonempty) set  
\[
A:= \big\{
\alpha \, \big|\,
\exists \{x_i^{(\alpha)}\}_{i\in \N} \subset 
\R \setminus \{0\} \ \text{such that} \ 
\displaystyle{\lim_{i \to \infty} x_i^{(\alpha)}=0, \ \lim_{i\to\infty} g'(x_i^{(\alpha)})=\alpha
}
\big\}
\] 
For instance, in the case of $\alpha = -1/2$, 
set $1/x_i^{(-1/2)} := \pi/3 + 2(i-1)\pi$ 
($i \in \N$). We then have 
$\lim_{i\to\infty}x_i^{(-1/2)} =0$ and 
$
\lim_{i\to\infty}g'(x_i)
= -1/2$. 
So $A= \Conv (A)=[-1,1]$ (cf.\,\cite[Theorem 2.1]{Ro}). 
Since $0\in \Conv (A)$, 
we can regard $x=0$ as a 
critical point of $g$. Clarke called 
$\Conv (A)$ the {\em generalized differential} of $g$.\end{enumerate} 

\subsection{An intrinsic definition of the generalized differential of Lipschitz maps}\label{2018_sec2.1}

The aim of this subsection is to 
intrinsically define the generalized differential for Lipschitz 
maps between Riemannian manifolds and their 
singular points in the sense of Clarke.\par 
We first recall Whitehead's convexity theorem. 
The theorem not only allows us to intrinsically 
define the generalized differential for Lipschitz maps between Riemannian manifolds, but also plays an important 
role in our smooth approximation method for such maps. 
A proof of the theorem can be found 
in \cite{Wh} or \cite[Proposition 4.2, pp. 76--77]{dC}.

\begin{theorem}
{\bf (Whitehead's convexity theorem)}
\label{Whitehead}
Let $X$ be a Riemannian manifold and $d_X$ 
the distance function on $X$. Then for each $x \in X$ 
there is a constant $\alpha(x)>0$ such that 
\begin{enumerate}[{\rm (a)}] 
\item\label{2018_01_28_Sect2_a} 
the open ball 
$B_{\alpha(x)}(x):=\{y \in X\,|\, d_X(x, y) <\alpha (x)\}$ 
is strongly convex, i.e., 
for any two points $p,q \in X$ 
in the closure $\ol{B_{\alpha(x)}(x)}$ 
there is a unique geodesic segment 
$\gamma:[0,1]\to X$ emanating 
from $p = \gamma (0)$ to $q= \gamma (1)$ such that
$\gamma (0,1)\subset B_{\alpha(x)}(x)$; 
\item\label{2018_01_28_Sect2_b} 
the exponential map $\exp_{x}|_{\B_{\alpha (x)}(o_x)}
: \B_{\alpha (x)}(o_x) \to B_{\alpha (x)}(x)$ at $x$ 
is a diffeomorphism where 
$\B_{\alpha (x)}(o_x):=\{v \in T_x X\,|\, \| v\|<\alpha(x)\}$ 
and $o_x$ indicates the origin of the tangent space 
$T_xX$ at $x$. 
\end{enumerate}
\end{theorem}

From now on let $M$ and $N$ be Riemannian manifolds 
of dimension $m$ and $n$, respectively, and 
$F: M\to N$ a locally Lipschitz map. The following 
lemma is a direct consequence of Theorem \ref{Whitehead}.  

\begin{lemma}\label{2018_02_01_lem2.2}
For each $p \in M$ there are 
two open balls $B_{r (p)}(p) \subset M$ 
and $B_{t (p)}(F(p))\subset N$ such that 
\begin{enumerate}[{\rm (i)}]
\item\label{2018_02_01_lem2.2_i} 
both $B_{r (p)}(p)$ and $B_{t (p)}(F(p))$ 
satisfy \eqref{2018_01_28_Sect2_a} and \eqref{2018_01_28_Sect2_b} 
of Theorem \ref{Whitehead};
\item\label{2018_02_01_lem2.2_ii}
$F(B_{r (p)}(p)) \subset B_{t (p)}(F(p))$; 
\item\label{2018_02_01_lem2.2_iii} 
$F|_{B_{r (p)}(p)}: B_{r (p)}(p)\to B_{t (p)}(F(p))$ 
is Lipschitz continuous.
\end{enumerate} 
\end{lemma}

\medskip

Using parallel transport we intrinsically define 
the generalized differential for $F$: 

\begin{definition}\label{2018_02_12_def2.1}
We will use the following notation. If there exists a unique geodesic between $x,y\in M$, then denote parallel transport 
along that geodesic by $\tau_y^x : T_x M \rightarrow T_y M$.
\end{definition}

For each $x \in M$ let 
$\cL (T_x M, T_{F(x)}N)$ be 
the set of all linear mappings of $T_x M$ to $T_{F(x)}N$. 
Since $\cL (T_x M, T_{F(x)}N)$ 
is isomorphic to the vector space $\cM (n, m; \R)$ of $n\times m$-matrices with real entries, $\cL (T_x M, T_{F(x)}N)$ is 
an $nm$-dimensional vector space. 
We topologize $\cL (T_x M, T_{F(x)}N)$ 
with the operator norm $\|\,\cdot\,\|$, so that, 
throughout this article, 
we regard $\cL (T_x M, T_{F(x)}N)$ as 
a finite dimensional normed vector space.\par  
Fix $p \in M$. Choose the two balls
$B_{r (p)}(p) \subset M$ and $B_{t (p)}(F(p))\subset N$
satisfying all three properties 
\eqref{2018_02_01_lem2.2_i}--\eqref{2018_02_01_lem2.2_iii}  
of Lemma \ref{2018_02_01_lem2.2}.
From Rademacher's theorem \cite{R} 
there is a set $E_F\subset M$ of Lebesgue 
measure zero such that the differential $dF$ 
of $F$ exists on $B_{r (p)}(p) \setminus E_F$. 
Since $F$ is Lipschitz on $B_{r(p)}(p)$, 
and since for $q\in B_{r(p)}(p)$ parallel transports 
$\tau_q^p$ and $\tau_{F(p)}^{F(q)}$ are 
linear isometries on $B_{r(p)}(p)$ and $B_{t(p)}(F(p))$, respectively, 
$\{\tau_{F(p)}^{F(q)}\circ dF_{q} \circ \tau_{q}^{p}\}_{q \in B_{r (p)}(p) \setminus E_F}$ is bounded in $\cL (T_p M, T_{F(p)}N)$. 
Since $B_{r (p)}(p) \setminus E_F$ is dense in $B_{r (p)}(p)$, there is 
a sequence $\{x_i\}_{i \in \N} \subset 
B_{r (p)}(p) \setminus E_F$ 
such that 
$\lim_{i \to \infty} x_i=p$ and 
$\{
\tau_{F(p)}^{F(x_i)}\circ dF_{x_i} \circ \tau_{x_i}^{p}
\}_{i \in \N}
$ 
converges in $\cL (T_p M, T_{F(p)}N)$. 
Hence we can introduce the notion of 
the ``mixture" of the differential of $F$ 
as follows: 

\begin{definition}{\rm (compare \cite{KT})}\label{2018_02_01_def2.3}
For each $p \in M$ we call the set 
\begin{equation}\label{2018_01_31_cpt}
K_F (p)
:=
\Big\{
G \in \cL (T_p M, T_{F(p)}N) \Bigm|
\begin{array}{l}
\exists \{x_i\}_{i \in \N} \subset 
B_{r (p)}(p) \setminus E_F \ \  \text{such that}\\
\displaystyle{\lim_{i \to \infty} x_i=p, \ \lim_{i \to \infty} \tau_{F(p)}^{F(x_i)}\circ dF_{x_i} \circ \tau_{x_i}^{p}= G}
\end{array}
\Big\}
\end{equation}
the {\em mixture of the differential of $F$ at $p$}. 
Note here that, from 
\eqref{2018_02_01_lem2.2_ii} 
of Lemma \ref{2018_02_01_lem2.2}, 
parallel transport 
$\tau_{F(p)}^{F(x_i)}: T_{F(x_i)}N\to T_{F(p)}N$ 
can be defined.
\end{definition}

\begin{remark}\label{2020_01_07_rem2.5}
By Definition \ref{2018_02_01_def2.3}, 
for any $p \in M$, $K_F (p)$ is a nonempty 
bounded set in $\cL (T_p M, T_{F(p)}N)$.
\end{remark}

The generalized differential for $F:M\to N$ 
is now intrinsically defined as follows: 

\begin{definition}\label{2018_01_09_def2.2}
For each $p \in M$ we call the set 
$\partial F (p):= \Conv (K_F (p))$ 
the {\em generalized differential of $F$ at $p$} 
where again $\Conv (K_F (p))$ denotes 
the convex hull of the mixture $K_F (p)$ 
of the differential of $F$ at $p$.
\end{definition}

\begin{remark}\label{2018_01_12_rem2.4} 
We give remarks on 
Definition \ref{2018_01_09_def2.2}: 
Fix $p \in M$. 
\begin{enumerate}[{\rm (i)}]
\item 
Clarke \cite{C2} originally called $\partial F (p)$ 
the {\it generalized Jacobian of $F$ at $p$} 
where $M$ and $N$ 
are Euclidean spaces 
of the same dimension $m$. 
This is because we can use the atlas 
$\{(\R^m, \id_{\R^m})\}$ with a single chart on $\R^m$ 
where $\id_{\R^m}:\R^m\to\R^m$ is the identity 
map, so without referring to independence of the choice of charts we can define $\partial F(p)$ as follows: 
\begin{equation}\label{2018_06_24_rem2.4_1} 
\partial F (p):=\Conv 
\Big(
\Big\{
A \in \cM (n, m; \R) \Bigm|
\begin{array}{l}
\exists \{x_i\}_{i \in \N} \subset 
\R^m \setminus E_F \ \  \text{such that}\\
\displaystyle{\lim_{i \to \infty} x_i=p, \ \lim_{i \to \infty} (JF)_{x_i}= A}
\end{array}
\Big\}
\Big)
\end{equation}
where $(JF)_{x_i}$ 
indicates the Jacobian matrix of $F$ at $x_i$.
\item 
From \cite[Theorem 17.2]{Ro}, $\partial F (p)$ is 
a compact convex subset of $\cL (T_p M, T_{F(p)}N)$. 
\item Although the following fact was mentioned 
in Section \ref{2020_09_12_sub2.1} (ii), we will 
mention it again: \cite[Theorem 2.1]{Ro} shows that 
$\partial F(p)$ is the smallest convex set containing 
$K_F(p)$, i.e., 
\[
\partial F(p) = \Conv (K_F(p))
= \bigcap \{Z\,|\, Z \ \text{is convex in} \ 
\cL (T_pM, T_{F(p)}N) \ \text{with}\ K_F(p) \subset Z\}.
\]
\item Since $\dim \cL (T_pM, T_{F(p)}N) = nm$, 
it follows from Carath\'eodory's theorem 
(cf.\,\cite[Theorem 1.1.4]{Sch}) that 
for any $g \in \partial F (p)$ there are 
$G_1, G_2, \dots, G_{nm+1} \in K_F (p)$ such that 
$g= \sum_{i=1}^{nm+1}a_i G_i$ 
where $\sum_{i=1}^{nm+1}a_i =1$ and $a_i \ge 0$ 
($i =1,2,\dots, nm+1$).
\item If $dF_p$ exists, 
then $dF_p \in \partial F(p)$. Moreover, if $F$ is 
of class $C^1$ on $B_{r(p)}(p)$, 
then $\partial F(p)$ is a singleton, 
which means $\partial F(p)=  \{dF_p\}$.
\item As a direct consequence of 
Definition \ref{2018_01_09_def2.2}
we observe that for any $\ve >0$ there is a constant 
$\mu (p, \ve)\in (0, r(p))$ such that 
$\tau^{F(x)}_{F(p)}\circ \partial F(x) \circ \tau^p_x \subset \cU_\ve (\partial F(p))$ 
for all $x \in B_{\mu (p,\,\ve)} (p)$ where 
$\tau^{F(x)}_{F(p)}\circ \partial F(x) \circ \tau^p_x 
:=\{\tau^{F(x)}_{F(p)}\circ g \circ \tau^p_x\,|\, g \in \partial F(x)\}$ 
and $\cU_\ve (\partial F(p))$ denotes the $\ve$-open neighborhood of $\partial F(p)$ 
in $\cL (T_pM, T_{F(p)}N)$.
\end{enumerate}
\end{remark}

Now that we have defined the generalized differential 
for Lipschitz maps between Riemannian manifolds, 
it is time to intrinsically define their singular points: 

\begin{definition}\label{2018_01_10_def2.7}
A point $p \in M$ is said 
to be {\it nonsingular for $F$} (or a {\em nonsingular point of $F$}) {\em in the sense of Clarke} 
if every element in $\partial F(p)$ is of maximal rank, 
i.e., for any $g \in \partial F (p)$, 
$\ra (g) = \min \{m, n\}$. 
\end{definition}

\begin{remark}
Clarke \cite{C2} first introduced the notion of singular points of 
Lipschitz maps between Euclidean spaces 
of the same dimension. Using this notion 
he extended 
the inverse function theorem for smooth maps between 
Euclidean spaces to Lipschitz ones, see \cite[Theorem 1]{C2}. 
\end{remark}

\subsection{The relationship between the 
generalized differential and the generalized 
gradient of Lipschitz functions}\label{2018_sec2.2}

In this subsection 
we define the generalized gradient of 
Lipschitz functions on Riemannian manifolds 
and study the relationship between their generalized gradient 
and their generalized differential. 
Throughout this subsection 
let $M$ be a Riemannian manifold of dimension $m$ 
with Riemannian metric $\langle \, \cdot\,, \cdot\, \rangle$, 
$F: M\to \R$ a locally Lipschitz function, and $t$ the standard coordinate on $\R$.\par 
Fix $p \in M$. By Theorem \ref{Whitehead} 
there is a strongly convex open ball 
$B_{r (p)}(p)$ such that 
$\exp_p|_{\B_{r(p)}(o_p)}: \B_{r(p)}(o_p) \to B_{r (p)}(p)$ 
is a diffeomorphism. For each $x \in B_{r (p)}(p)$ 
parallel transport $\tau^p_{x}:T_p M \to T_xM$ is 
defined as in Definition \ref{2018_02_12_def2.1}. 
Note that parallel transport 
$\tau^x_{y}:T_x \R \to T_y\R$ is defined for all $x,y \in \R$. 
Let $K_F(p)$ be the mixture of the differential of $F$ 
at $p$ defined by Eq.\,\eqref{2018_01_31_cpt} 
for $T_{F(p)}N=T_{F(p)}\R$, and $E_F$ a set of Lebesgue measure zero such that $dF$ exists on $B_{r (p)}(p) \setminus E_F$. 
The gradient vector field of $F$ denoted by $\grad F$ 
can be defined on $B_{r (p)}(p) \setminus E_F$ 
because $F$ is differentiable there. Note that 
\begin{equation}\label{2020_01_09_grad}
\langle (\grad F)_x, u \rangle 
\cdot \frac{d}{dt}\bigg|_{F(x)}
= u(F)\cdot \frac{d}{dt}\bigg|_{F(x)}
= dF_x(u)
\end{equation}
for all $x \in B_{r (p)}(p) \setminus E_F$ and $u \in T_xM$.

\begin{lemma}\label{2018_07_08_lem2.17}
For any $G \in K_F (p)$ there is 
a sequence 
$\{x_i\}_{i\in\N} \subset B_{r(p)}(p)\setminus E_F$ 
such that $\lim_{i\to\infty}x_i=p$, 
$\lim_{i\to\infty} \tau_{F(p)}^{F(x_i)}\circ dF_{x_i}\circ \tau_{x_i}^p =G$, and 
\[
G(v)
= \lim_{i \to \infty} 
\langle
(\grad F)_{x_i}, \tau_{x_i}^p(v)
\rangle 
\frac{d}{dt}\bigg|_{F(p)} \qquad (v \in T_pM).
\]
\end{lemma}

\begin{proof}Fix $G \in K_F(p)$. The definition of 
$K_F(p)$ gives a sequence 
$\{x_i\}_{i\in\N} \subset B_{r(p)}(p)\setminus E_F$ 
satisfying 
$\lim_{i\to\infty}x_i=p$ 
and 
$\lim_{i\to\infty} \tau_{F(p)}^{F(x_i)}\circ dF_{x_i}\circ \tau_{x_i}^p =G$. For any $v \in T_pM$, 
Eq.\,\eqref{2020_01_09_grad} gives 
\begin{align}\label{2018_07_08_lem2.17_1}
G(v)
&=
\Big(
\lim_{i\to \infty}
\tau_{F(p)}^{F(x_i)}\circ dF_{x_i}\circ \tau_{x_i}^p
\Big)
(v)
=
\lim_{i \to \infty}
\tau^{F(x_i)}_{F(p)} 
\Big(
\tau_{x_i}^p(v) (F)\frac{d}{dt}\bigg|_{F(x_i)}
\Big)
\\
&=
\lim_{i \to \infty}
\tau_{x_i}^p(v) (F)
\frac{d}{dt}\bigg|_{F(p)}
=
\lim_{i \to \infty} 
\langle
(\grad F)_{x_i}, \tau_{x_i}^p(v)
\rangle 
\frac{d}{dt}\bigg|_{F(p)}.\notag
\end{align}
$\qedd$
\end{proof}

\begin{definition}\label{2020_01_09_def2.12}
Eq.\,\eqref{2018_07_08_lem2.17_1} defines 
$\lim_{i \to \infty} (\grad F)_{x_i}$, that is, 
for each $G \in K_F (p)$ let 
\[
\langle
\lim_{i \to \infty}(\grad F)_{x_i}, v
\rangle 
\frac{d}{dt}\bigg|_{F(p)}
:= G(v) \qquad (v \in T_pM).
\]
\end{definition}

The following lemma follows directly 
from Lemma \ref{2018_07_08_lem2.17} together with Definition \ref{2020_01_09_def2.12}.

\begin{lemma}\label{2020_01_09_lem2.13}
The 
set 
\[
\divi_F(p):=
\Big\{
w \in T_p M \Bigm|
\begin{array}{l}
\exists \{x_i\}_{i \in \N} \subset 
B_{r (p)}(p) \setminus E_F \ \  \text{such that}\\
\displaystyle{\lim_{i \to \infty} x_i=p, \ 
\lim_{i \to \infty} (\grad F)_{x_i}= w}
\end{array}
\Big\}
\]
is nonempty.
\end{lemma}

\begin{definition}\label{2018_07_08_def2.18}
We call $\divi_F(p)$ 
the {\em mixture of the gradient} of $F$ at $p$, and 
the convex set $\cdas_F(p):= \Conv (\divi_F(p))$ 
the {\em generalized gradient} of $F$ at $p$. 
\end{definition}

\begin{remark}\label{2020_01_13_rem2.15} 
$\cdas_F (p)$ 
is compact in $T_p M$ by \cite[Theorem 17.2]{Ro}. 
Note that Clarke \cite{C1} first defined 
the generalized gradient of Lipschitz functions 
on $\R^m$.
\end{remark}

\begin{lemma}\label{2018_07_08_lem2.20}
Let $\partial F(p)$ be 
the generalized differential of $F$ at $p$. Then  
for any $g \in \partial F(p)$ there is a vector 
$X^{(g)} \in \cdas_F(p)$ such that 
$g(v) = \langle X^{(g)}, v \rangle \cdot d/dt|_{F(p)}$ 
for all $v \in T_pM$.
\end{lemma}

\begin{proof}Fix $g \in \partial F(p)$. 
By Carath\'eodory's theorem (cf.\,\cite[Theorem 1.1.4]{Sch}) 
there are vectors $G_1, G_2, \ldots, G_{m+1} \in K_F(p)$ 
such that $g=\sum_{k=1}^{m+1} a_k G_k$ where 
$\sum_{k=1}^{m+1} a_k=1$
and 
$a_k \ge 0$ ($k =1,2,\ldots, m +1$). By Lemmas 
\ref{2018_07_08_lem2.17} 
and \ref{2020_01_09_lem2.13} for each 
$k =1,2,\ldots, m +1$ 
there is a vector $w^{(k)} \in \divi_F(p)$ 
such that 
$
G_k (v)
= \langle w^{(k)}, v \rangle (d/dt)|_{F(p)}
$ 
for all $v \in T_pM$, and hence 
\[
g(v)
=
\sum_{k=1}^{m +1} a_k G_k (v)
=
\sum_{k=1}^{m +1} a_k \langle w^{(k)}, v \rangle \frac{d}{dt}\bigg|_{F(p)}
=
\Big
\langle \sum_{k=1}^{m +1} a_k w^{(k)}, v 
\Big\rangle \frac{d}{dt}\bigg|_{F(p)}.
\]
Since $\sum_{k=1}^{m +1} a_k w^{(k)} \in \cdas_F(p)$, 
$X^{(g)}:= \sum_{k=1}^{m +1} a_k w^{(k)}$ is 
the desired vector.
$\qedd$
\end{proof}

\begin{lemma}\label{2018_07_08_lem2.21}
$p$ is singular for $F$ if and only if $o_p \in \cdas_F(p)$. 
\end{lemma}

\begin{proof}
We first assume that $p$ is singular for $F$. 
There is then a point $g_0 \in \partial F(p)$ such that 
$\ra (g_0) = 0$. 
By Lemma \ref{2018_07_08_lem2.20} 
there is a vector $X^{(0)} \in \cdas_F (p)$ such that 
$g_0(v) = \langle X^{(0)}, v \rangle (d/dt)|_{F(p)}$ 
holds for all $v \in T_pM$. Since $\ra (g_0) = 0$, 
$o_{F(p)} =\langle X^{(0)}, v \rangle \cdot d/dt |_{F(p)}$ 
for all $v \in T_pM$ where $o_{F(p)}$ indicates 
the origin of $T_{F(p)}\R$, hence $X^{(0)} =o_p$, and finally 
$o_p= X^{(0)} \in \cdas_F(p)$.\par 
We next assume $o_p \in \cdas_F(p)$. 
Carath\'eodory's theorem (cf.\,\cite[Theorem 1.1.4]{Sch}) 
shows that there are vectors 
$w_1,w_2,\ldots,w_{m+1} \in \divi_F(p)$ 
such that 
$o_p = \sum_{k=1}^{m+1} \alpha_k w_k$ 
where $\sum_{k=1}^{m+1} \alpha_k =1$ and 
$\alpha_k \ge 0$ ($k= 1, 2, \ldots, m+1$). 
Define the linear map $f_0: T_pM\to T_{F(p)}\R$ by 
$
f_0(v)
:=
\langle 
\sum_{k=1}^{m+1} \alpha_k w_k, v
\rangle 
\cdot d / dt |_{F(p)}
$
for all $v \in T_p M$. 
We observe $f_0 \in \partial F(p)$. 
Since $f_0(v)= o_{F(p)}$ for all $v \in T_pM$, 
$\ra (f_0)=0$, and hence 
$p$ is singular for $F$.$\qedd$
\end{proof}

\begin{remark}
In \cite{C1} and \cite{KT} a point $p \in M$ 
is called 
{\em noncritical for $F$} if $o_p \not\in \cdas_F(p)$.\end{remark}

\subsection{Critical points of distance functions in the sense of Grove--Shiohama are singular points of Clarke.}\label{2018_sec2.4}

Throughout this subsection let $M$ be a complete Riemannian manifold of dimension $m$ with Riemannian metric 
$\langle\,\cdot\,, \,\cdot\, \rangle$ and 
the distance function $d$. All geodesics will be normal.\par 
Fix $p \in M$. Define the map $d_p:M \to \R$ by 
$d_p (x):= d(p, x)$ for all $x \in M$. 
We then have the following proposition. Note that 
the proposition appeared as \cite[Example 1.9]{KT}; however we are sometimes asked the proof, so that we 
give the details here.

\begin{proposition}\label{2020_01_10_prop2.19}
$q \in M$ is singular for $d_p$ in the sense of Clarke 
if and only if $q$ 
is critical for $d_p$ in that 
of Grove--Shiohama.
\end{proposition}

\begin{proof}
For each $x \in M$ let $B_{r(x)}(x)$ be 
a strongly convex open ball, 
guaranteed by Theorem \ref{Whitehead}, 
such that   
$\exp_x|_{\B_{r(x)}(o_x)}$ is a diffeomorphism, and 
$E_{d_p}$ a set of Lebesgue measure zero 
such that the differential of $d_p$ exists on 
$M \setminus E_{d_p}$. 
This proposition follows from Lemmas \ref{2018_07_09_lem2.27} and \ref{2018_07_11_lem2.28} 
below.$\qedd$
\end{proof}

\begin{lemma}\label{2018_07_09_lem2.27}
$q \in M\setminus \{p\}$ is singular for $d_p$ in the sense of Clarke 
if and only if $q$ 
is critical for $d_p$ in that 
of Grove--Shiohama.
\end{lemma}

\begin{proof}
Assume that 
$q \in M\setminus \{p\}$ is singular 
for $d_p$ in the sense of Clarke. 
By Lemma \ref{2018_07_08_lem2.21}, 
$o_q \in \cdas_{d_p}(q)$ holds where 
$\cdas_{d_p}(q)$ indicates the generalized gradient of 
$d_p$ at $q$. From Carath\'eodory's theorem 
(cf.\,\cite[Theorem 1.1.4]{Sch}) 
there are $w_1,w_2,\ldots,w_{m+1} \in \divi_{d_p}(q)$ 
such that 
$o_q = \sum_{k=1}^{m+1} \alpha_k w_k$ 
where $\sum_{k=1}^{m+1} \alpha_k =1$ 
and 
$\alpha_k \ge 0$ ($k= 1,2, \ldots, m+1$), 
and $\divi_{d_p}(q)$ denotes 
the mixture of the gradient of $d_p$ at $q$. 
Fix $v \in T_qM \setminus \{o_q\}$. We then have 
$0= \langle \sum_{k=1}^{m+1} \alpha_k w_k, v \rangle
=
\sum_{k=1}^{m+1} \alpha_k \langle w_k, v \rangle$. 
Since $\alpha_k \ge 0$ ($k= 1,2, \ldots, m+1$), 
there is a number $k_0\in \{1,2, \ldots,m+1\}$ 
such that 
$\langle w_{k_0}, v \rangle \le 0$. 
As $w_{k_0} \in \divi_{d_p}(q)$, 
there is a sequence 
$
\{x_i^{(k_0)}\}_{i\in \N} 
\subset B_{r(q)}(q) \setminus E_{d_p}
$ 
such that $\lim_{i\to\infty}x_i^{(k_0)}=q$ and 
$\lim_{i\to\infty}(\grad d_p)_{x_i^{(k_0)}} =w_{k_0}$. 
Now $M$ is complete, 
so for each $i \in \N$ there is 
a minimal geodesic segment 
$\gamma_i:[0, d_p(x_i^{(k_0)})]\to M$ emanating from 
$p$ to $x_i^{(k_0)}$, and hence we obtain the sequence 
$\{\gamma_i\}_{i \in \N}$ of such geodesics. 
The set 
$\Sph_p^{m-1}:=\{u \in T_pM\,|\,\|u\|=1\}$ 
is compact, so we can assume, 
by taking a subsequence of 
$\{(d\gamma_i/dt) (0)\}_{i \in \N} \subset \Sph_p^{m-1}$ 
if necessary, that 
$\lim_{i \to \infty}(d\gamma_i/dt) (0) \in \Sph_p^{m-1}$ exists. Set $u:= \lim_{i \to \infty}(d\gamma_i/dt) (0)$. 
Since $\lim_{i\to \infty} x_i^{(k_0)} = q$, 
$\{\gamma_i\}_{i \in \N}$ converges 
to a minimal geodesic segment 
$\gamma_\infty :[0, d_p(q)]\to M$ emanating from 
$p$ to $q$ given by $\gamma_\infty (t)=\exp_p tu$.
Moreover, from 
\cite[Proposition 4.8 of Chap.\,III]{Sak} 
we have 
$(\grad d_p)_{x_i^{(k_0)}} 
= (d\gamma_i / dt) (d_p(x_i^{(k_0)}))$ for 
all $i \in \N$. 
Note that $\|(\grad d_p)_{x_i^{(k_0)}} \|=1$ 
for each $i\in\N$. Since 
$
w_{k_0}= \lim_{i\to \infty} (\grad d_p)_{x_i^{(k_0)}} 
= (d\gamma_\infty / dt) (d_p(q))$ and $\langle w_{k_0}, v \rangle \le 0$, 
we see that 
$
0 \ge \langle w_{k_0}, v \rangle 
= 
\|v\|\cos \angle 
( (d\gamma_\infty / dt) (d_p(q)), v)$, 
hence 
$
\angle (- (d\gamma_\infty / dt) (d_p(q)), v) 
\le \pi / 2$ 
holds for all $v \in T_qM$, and finally 
$q$ is therefore critical for $d_p$ in 
the sense of Grove--Shiohama.\par 
We next assume that $q$ is critical for $d_p$ in 
the sense of Grove--Shiohama. 
Fix $v \in \Sph_q^{m-1}:=\{u \in T_qM\,|\,\|u\|=1\}$. 
There is a minimal geodesic segment $\sigma^{(v)}:[0, d_p(q)]\to M$ emanating from 
$p$ to $q$ such that 
$\angle 
(-(d\sigma^{(v)} / dt) (d_p(q)), v) \le \pi / 2$. 
As $\exp_q|_{\B_{r(q)}(o_q)}$ is a diffeomorphism 
onto $B_{r(q)}(q)$, we have a unique minimal geodesic 
$c_v :(-r(q), r(q)) \to B_{r(q)}(q)$ given by 
$c_v(s):= \exp_q sv$ for all $s \in (-r(q), r(q))$. 
Also $(\exp_q)^{-1}$ is a diffeomorphism 
from $B_{r (q)}(q)$ onto $\B_{r(q)}(o_q)$, 
so it follows from \cite[Lemma 6.5]{WhZy} that 
we can choose a 
sequence $\{s_i\}_{i \in \N} \subset (-r(q), r(q))$ 
such that $\lim_{i\to\infty}s_i=0$ and 
$c_v(s_i) \in B_{r (q)}(q) \setminus E_{d_p}$. 
To simplify notation we set 
$y_i:= c_v(s_i)$ for each $i \in \N$. 
Note that $\lim_{i\to\infty}y_i=q$. 
Since $M$ is complete, 
for each $i \in \N$ there is 
a minimal geodesic segment 
$\eta_i:[0, d_p(y_i)]\to M$ emanating from 
$p$ to $y_i$.
By the same argument above we can assume that 
$\{\eta_i\}_{i \in \N}$ converges 
to a minimal geodesic segment 
$\eta_\infty :[0, d_p(q)]\to M$ emanating from 
$p$ to $q$. From 
\cite[Proposition 4.8 of Chap.\,III]{Sak}, 
$(\grad d_p)_{y_i} 
= (d\eta_i / dt) (d_p(y_i))$ 
holds for each $i \in \N$. 
Now 
$
\lim_{i\to \infty}(\grad d_p)_{y_i} 
= (d\eta_\infty / dt) (d_p(q))
$, 
set $w^{(v)}:= (d\eta_\infty/dt) (d_p(q))$, and 
$
w^{(v)} 
= 
\lim_{i\to \infty}(\grad d_p)_{y_i} \in \divi_{d_p} (q)$. 
Since $\angle 
(- (d\sigma^{(v)} / dt) (d_p(q)), v) \le  \pi / 2$, 
\cite[Lemma 2.1]{IT} shows 
$
\angle (-w^{(v)}, v) 
= 
\angle (-\lim_{i\to \infty}(\grad d_p)_{y_i}, v) 
\le \pi / 2$, 
and hence we get 
$\angle (w^{(v)}, v) \ge \pi / 2$. 
Since $v \in \Sph_q^{m-1}$ was arbitrary, for each $v \in \Sph_q^{m-1}$ we take
$w^{(v)} \in \divi_{d_p}(q)$ satisfying 
$\angle (w^{(v)}, v) \ge \pi / 2$, 
and we set $W:=\{w^{(v)}\in \divi_{d_p}(q)\,|\,v \in \Sph_q^{m-1}\}$. 
Then $W$ is not contained in an open half space of 
$T_qM$, which implies 
$o_q \in \Conv(W)\subset \cdas_{d_p}(q)$. 
Lemma \ref{2018_07_08_lem2.21} gives 
that $q$ is singular for $d_p$.$\qedd$
\end{proof}

\begin{lemma}\label{2018_07_11_lem2.28}
$p$ is also singular for $d_p$ in the sense 
of Clarke.
\end{lemma}

\begin{proof}
Note that $d_p$ is differentiable 
on $B_{r(p)} (p)\setminus \{p\}$, for 
the set has no cut points of $p$. 
We first show that $\divi_F(p) = \Sph_p^{m-1}$. 
Indeed, since all the geodesics are
normalized, \cite[Proposition 4.8 of Chap.\,III]{Sak} 
gives $\divi_{d_p}(p) \subset \Sph_p^{m-1}$. Thus it is sufficient to prove 
$\Sph_p^{m-1} \subset \divi_{d_p}(p)$. Fix 
$v \in \Sph_p^{m-1}$. 
Let $\sigma_v :(-r(p),r(p))\to B_{r(p)}(p)$ 
be a minimal geodesic defined by 
$\sigma_v(t):= \exp_p tv$ 
for all $t \in (-r(p),r(p))$. 
Let $\{t_i\}_{i\in \N}$ be a sequence 
of constants 
$t_i \in (-r(p),r(p))\setminus \{0\}$ 
converging to $0$ by letting $i \to \infty$. 
Set $x_i:= \sigma_v(t_i)$ for each $i \in \N$. 
Note that $x_i \in B_{r(p)} (p)\setminus \{p\}$. 
Combining the Gauss lemma (cf.\,\cite[(1) of 
Proposition 2.3 of Chap.\,III]{Sak}) 
and \cite[Proposition 4.8 of Chap.\,III]{Sak} 
gives that 
$
(\grad d_p)_{\sigma_v (t_i)} = (d\sigma_v / dt)(t_i) 
= (d \exp_p)_{t_iv}v$. 
Since 
$
\lim_{i \to \infty}(\grad d_p)_{\sigma_v (t_i)} 
= \lim_{i \to \infty} (d \exp_p)_{t_iv}v 
= (d \exp_p)_{o_p}v 
=v
$, 
we get $v \in \divi_{d_p}(p)$, i.e., 
$\Sph_p^{m-1} \subset \divi_{d_p}(p)$ holds. Therefore 
$\divi_{d_p}(p)= \Sph_p^{m-1}$.\par  
Since $\divi_{d_p}(p)= \Sph_p^{m-1}$, 
$\cdas_{d_p}(p) = \Conv (\Sph_p^{m-1}) 
= \{X \in T_pM\,|\, \|X\|\le1\}$ holds, 
hence $o_p \in \cdas_{d_p}(p)$, and finally 
Lemma \ref{2018_07_08_lem2.21} shows that $p$ 
is singular for $d_p$ in that sense.$\qedd$
\end{proof}

\section{The adjoint of the generalized differential of 
Lipschitz maps}\label{2018_01_18_Sect3}

In this section we define the adjoint of the generalized differential of Lipschitz maps between Riemannian manifolds, study surjectivity and injectivity of the generalized differential near their 
nonsingular points, and finally show that the set of all singular points of the map is closed.

\subsection{Definition of the adjoint of the generalized differential}\label{2018_01_31_Sect3.2}

In this subsection we formulate the notion of the adjoint of the generalized differential of Lipschitz maps between Riemannian manifolds. Throughout this subsection 
let $M$ and $N$ 
be Riemannian manifolds 
with dimension $m$ and $n$ and 
Riemannian metrics 
$\langle \,\cdot\,, \,\cdot\,\rangle_M$ and 
$\langle \,\cdot\,, \,\cdot\,\rangle_N$, respectively. 
$F: M\to N$ is a locally Lipschitz map, 
and $K_F(p)$ the mixture of the differential of $F$ 
at $p \in M$.\par  
Fix $p\in M$. Choose 
two strongly convex open balls
$B_{r (p)}(p) \subset M$ and 
$B_{t (p)}(F(p))\subset N$ 
satisfying all three 
properties \eqref{2018_02_01_lem2.2_i}--\eqref{2018_02_01_lem2.2_iii}  
in Lemma \ref{2018_02_01_lem2.2}. Let 
$\cL(T_{F(p)}N, T_pM)$ be 
the $mn$-dimensional vector space 
of all linear mappings of $T_{F(p)} N$ to $T_pM$ 
topologized with the operator norm $\|\,\cdot\,\|$. 
Consider two nonempty sets  
\[
\adj(K_F(p))
:=\{G^* \in \cL(T_{F(p)}N, T_pM)\,|\, G \in K_F(p)
\}
\]
where $G^*$ denotes the adjoint of $G$, and 
\[
\{K_F(p)\}^*:=
\Big\{
H^* \in \cL (T_{F(p)}N, T_p M) \Bigm|
\begin{array}{l}
\exists \{x_i\}_{i \in \N} \subset 
B_{r (p)}(p) \setminus E_F \ \  \text{such that}\\
\displaystyle{\lim_{i \to \infty} x_i=p, \ \lim_{i \to \infty} \tau_{p}^{x_i}\circ (dF_{x_i})^* \circ \tau_{F(x_i)}^{F(p)}= H^*}
\end{array}
\Big\}
\]
where 
$\tau_{p}^{x_i}$ and $\tau_{F(x_i)}^{F(p)}$ are 
parallel transports as in 
Definition \ref{2018_02_12_def2.1}, 
and $(dF_{x_i})^*$ denotes the adjoint 
of $dF_{x_i}$ at each $x_i$. 
Note that, from the \eqref{2018_02_01_lem2.2_ii} 
of Lemma \ref{2018_02_01_lem2.2}, 
each 
$\tau_{F(x_i)}^{F(p)}: T_{F(p)}N\to T_{F(x_i)}N$ 
is defined in that sense.

\begin{lemma}\label{2018_02_12_lem3.6}
$\adj(K_F(p))= \{K_F(p)\}^*$.
\end{lemma}

\begin{proof}Fix $G^* \in \adj(K_F(p))$. 
Since $(G^*)^* = G \in K_F(p)$ by definition of $\adj(K_F(p))$, there is a sequence 
$\{x_i\}_{i \in \N} \subset B_{r (p)}(p) \setminus E_F$ 
such that 
$\lim_{i \to \infty} x_i=p$ and 
$\lim_{i \to \infty} \tau_{F(p)}^{F(x_i)}\circ dF_{x_i} \circ \tau_{x_i}^{p}= G$. Note that the adjoints of $\tau_{x_i}^{p}$ and $\tau_{F(p)}^{F(x_i)}$ are their 
inverses $\tau^{x_i}_{p}$ and $\tau^{F(p)}_{F(x_i)}$ 
since parallel transport is an isometry. 
Fix $u \in T_pM$ and $v \in T_{F(p)}N$. 
The Riesz representation 
theorem (cf.\,\cite[Theorem 10.1]{Rom}) 
then gives
\begin{align*}
\langle u, G^* (v)\rangle_M 
= \langle (\lim_{i \to \infty} \tau_{F(p)}^{F(x_i)}\circ dF_{x_i} \circ \tau_{x_i}^{p})(u), v\rangle_N 
= 
\langle 
u, 
(\lim_{i \to \infty} \tau_{p}^{x_i}\circ (dF_{x_i})^* \circ \tau_{F(x_i)}^{F(p)})(v)
\rangle_M,
\end{align*}
which implies 
$G^* \in \{K_F(p)\}^*$, 
and hence 
$\adj(K_F(p))\subset \{K_F(p)\}^*$. 
Our next claim is that $\{K_F(p)\}^* \subset  \adj(K_F(p))$. 
Indeed, for any fixed $H^* \in \{K_F(p)\}^*$ 
there is a sequence $\{x_i\}_{i \in \N} \subset 
B_{r (p)}(p) \setminus E_F$ such that 
$\lim_{i \to \infty} x_i=p$ and that 
$\lim_{i \to \infty} \tau_{p}^{x_i}\circ (dF_{x_i})^* \circ \tau_{F(x_i)}^{F(p)}= H^*$. For any $u \in T_pM$ and any $v \in T_{F(p)}N$ we have  
\[
\langle u, H^* (v)\rangle_M 
= \langle u, (\lim_{i \to \infty} \tau_{p}^{x_i}\circ (dF_{x_i})^* \circ \tau_{F(x_i)}^{F(p)})(v)\rangle_M
= \langle 
u, (\lim_{i \to \infty} \tau_{F(p)}^{F(x_i)}\circ dF_{x_i}
\circ \tau_{x_i}^{p})^*(v) \rangle_M, 
\]
which implies that 
$
H^* = (\lim_{i \to \infty} \tau_{F(p)}^{F(x_i)}\circ dF_{x_i}\circ \tau_{x_i}^{p})^*$. 
Since 
the adjoint 
$(H^*)^*$ of $H^*$ is unique 
by the Riesz representation 
theorem (cf.\,\cite[Theorem 10.1]{Rom}), 
$\lim_{i \to \infty} \tau_{F(p)}^{F(x_i)}\circ dF_{x_i}
\circ \tau_{x_i}^{p} = (H^*)^* \in K_F(p)$, 
and hence $\{K_F(p)\}^* \subset \adj(K_F(p))$. Therefore 
$\adj(K_F(p))= \{K_F(p)\}^*$.$\qedd$
\end{proof}

\begin{lemma}\label{2018_02_12_lem3.7}
Let 
$\adj (\partial F(p)):= \left\{
g^* \in \cL(T_{F(p)}N, T_pM)\,|\,g \in \partial F(p)
\right\}$ where $\partial F(p)$ is 
the generalized differential of $F$ at $p$ and 
$g^*$ denotes the adjoint of each $g \in \partial F(p)$. Then
\begin{enumerate}[{\rm (i)}] 
\item\label{2018_05_26_lem3.7_1}
$\adj (\partial F(p))$ is a nonempty and compact subset of 
$\cL(T_{F(p)}N, T_pM)$;
\item\label{2018_05_26_lem3.7_2}
$\adj (\partial F(p))= \Conv (\{K_F(p)\}^*)$.
\end{enumerate}
\end{lemma}

\begin{proof}We first show \eqref{2018_05_26_lem3.7_1}. Since $\partial F(p)\not= \emptyset$, 
the Riesz representation theorem (cf.\,\cite[Theorem 10.1]{Rom}) guarantees $\adj (\partial F(p))\not= \emptyset$. 
Take any sequence 
$\{g_i^*\}_{i\in \N} \subset \adj (\partial F(p))$ 
where each $g_i^*$ is the adjoint 
of $g_i \in \partial F(p)$. 
Since $\partial F(p)$ is compact, 
the sequence $\{g_i\}_{i\in \N}\subset \partial F(p)$ 
contains a subsequence $\{g_{i_k}\}_{k \in \N}$ 
which converges to some point $h \in \partial F(p)$ 
as $k\to\infty$. 
For any $u \in T_pM$ and any $v \in T_{F(p)}N$ we have  
$
\langle u, h^* (v)\rangle_M 
= \langle h(u), v\rangle_N
= \langle (\lim_{k \to \infty} g_{i_k})(u), v\rangle_N
= \langle u, (\lim_{k \to \infty} g_{i_k}^*)(v)\rangle_M$, 
hence $\lim_{k \to \infty} g_{i_k}^*= h^* \in \adj (\partial F(p))$. 
Since $\{g_{i_k}\}_{k \in \N} \subset \{g_i\}_{i \in \N}$, and 
since $g_{i_k}^*$ is the adjoint of $g_{i_k}$, 
$\{g_i^*\}_{i\in \N}$ contains $\{g_{i_k}^*\}_{k \in \N}$ as 
a subsequence converging 
to $h^* \in \adj (\partial F(p))$, 
which implies that $\adj (\partial F(p))$ is compact.\par 
Next we show \eqref{2018_05_26_lem3.7_2}. 
Let $g^* \in \adj (\partial F(p))$ 
where $g \in \partial F(p)$. From  
Carath\'eodory's theorem (cf.\,\cite[Theorem 1.1.4]{Sch}) 
there are $g_1, g_2,\ldots, g_{nm+1} \in K_F (p)$ such that 
$g^*= (\sum_{i=1}^{nm+1}a_i g_i)^*$ where 
$\sum_{i=1}^{nm+1} a_i =1$ and $a_i \ge 0$ ($i=1,2,\ldots, nm+1$). 
Since 
$g^*= \sum_{i=1}^{nm+1}a_i g_i^*$, 
and since Lemma \ref{2018_02_12_lem3.6} 
gives $g_i^* \in \adj(K_F (p))= \{K_F (p)\}^*$ 
($i=1,2,\ldots, nm+1$), 
$g^* \in \Conv (\{K_F(p)\}^*)$, i.e., $\adj (\partial F(p)) \subset \Conv (\{K_F(p)\}^*)$. The similar discussion shows $\Conv (\{K_F(p)\}^*) \subset \adj (\partial F(p))$, and hence $\adj (\partial F(p)) = \Conv (\{K_F(p)\}^*)$.
$\qedd$
\end{proof}

Lemma \ref{2018_02_12_lem3.7} justifies 
the following definition. 

\begin{definition}\label{2018_11_08_def3.8}
We call the set $\{\partial F (p)\}^* := \adj(\partial F(p))$ the {\em adjoint of $\partial F(p)$}.
\end{definition}

\begin{lemma}\label{2018_02_16_lem3.9}
$\{\partial F (p)\}^*$ is a nonempty, compact, and  
convex subset of $\cL(T_{F(p)}N, T_pM)$. 
\end{lemma}

\begin{proof}This statement follows from 
Lemma \ref{2018_02_12_lem3.7} and 
\cite[Theorem 17.2]{Ro}.$\qedd$
\end{proof}

\subsection{Surjectivity and injectivity of the generalized differential near a nonsingular point}\label{2018_11_10_Sect3.3}

All notation in this subsection is the same 
as defined in Section \ref{2018_01_31_Sect3.2}. 
First we show surjectivity of the generalized differential of $F$ near a nonsingular point when $m\ge n$, and next show injectivity when $m\le n$. 
Finally we see that the set of all singular points of $F$ is a closed set in $M$. 

\begin{lemma}\label{2018_02_13_lem3.10}
For any $p \in M$ and any $\ve >0$ there is a constant 
$\mu(p, \ve) \in (0, r(p))$ such that 
$
\tau^{x}_{p}\circ \{\partial F (x)\}^* \circ \tau^{F(p)}_{F(x)} 
\subset \cU_\ve ( \{\partial F (p)\}^* )$ 
for all $F(x) \in B_{t(p)} (F(p))$ 
($x \in B_{\mu(p,\,\ve) } (p)$) where 
$
\tau^{x}_{p}\circ \{\partial F (x)\}^* \circ \tau^{F(p)}_{F(x)}
:=\{
\tau^{x}_{p}\circ g^* \circ \tau^{F(p)}_{F(x)}\,|\, g^* 
\in \{\partial F (x)\}^*\}
$ 
and $\cU_\ve (\{\partial F (p)\}^*)$ 
denotes the $\ve$-open neighborhood of $\{\partial F (p)\}^*$ 
in $\cL (T_{F(p)}N, T_pM)$.
\end{lemma}

\begin{proof}
This is a direct consequence of Definition \ref{2018_11_08_def3.8}.$\qedd$
\end{proof}

\begin{proposition}\label{2018_02_12_prop3.11}
Assume $m \ge n$. 
If a point $p \in M$ is nonsingular for $F$, 
then there are two constants 
$\lambda (p) >0$ and 
$\delta (p)>0$ satisfying the following properties: 
\begin{enumerate}[{\rm (i)}] 
\item $B_{2 \lambda (p)} (p)$ satisfies \eqref{2018_01_28_Sect2_a} 
and \eqref{2018_01_28_Sect2_b} 
of Theorem \ref{Whitehead}; 
\item $F|_{B_{2 \lambda (p)} (p)}$ is 
a Lipschitz map from $B_{2 \lambda (p)} (p)$ 
into $B_{t (p)}(F(p))$; 
\item For any 
$u \in \Sph_{F(p)}^{n-1}:=\{w \in T_{F(p)} N\,|\,\|w\|=1\}$ 
and 
any $x \in B_{2 \lambda (p)}(p)$ 
there is a vector 
$
V_x^{(u)} 
\in 
\Sph_x^{m-1}:=\{v \in T_x M\,|\,\|v\|=1\}
$ such that 
$
\langle 
V_x^{(u)}, (g^* \circ \tau^{F(p)}_{F(x)}) (u)
\rangle_M \ge \delta (p)$ 
holds for all $g^* \in \{\partial F (x)\}^*$. 
In particular 
$\langle 
V_x^{(u)}, (dF_x)^* (\tau^{F(p)}_{F(x)} (u)) 
\rangle_M 
\ge \delta (p)$ 
for all $x \in B_{2 \lambda (p)} (p)\setminus E_F$; 
\item Every $x \in B_{2 \lambda (p)} (p)$ 
is nonsingular for $F$. 
\end{enumerate}
\end{proposition}

\begin{proof}
Fix $p \in M$ nonsingular for $F$. 
By Definition \ref{2018_01_10_def2.7}, 
$\ra (g) = n$ holds for all $g \in \partial F(p)$.
Since 
$\ra (g^*)= \ra (g) = n$ for all $g^*\in \{\partial F (p)\}^*$, 
$\{\partial F (p)\}^*$ has maximal rank. 
Every $g^*\in \{\partial F (p)\}^*$ is therefore injective.\par 
Take $u \in \Sph^{n-1}_{F(p)}$. Set 
$\{\partial F (p)\}^*u
:= \{ 
g^*(u)\,|\,g^*\in \{\partial F (p)\}^*
\} 
\subset T_{p}M$. 
Lemma \ref{2018_02_16_lem3.9} implies that 
$\{\partial F (p)\}^*u$ is compact and convex 
in $T_pM$. 
Since each $g^*\in \{\partial F (p)\}^*$ is injective 
as we have seen above, 
$o_p \not\in \{\partial F (p)\}^*u$ holds where 
$o_p$ indicates the origin of $T_pM$. 
Since $\{\partial F (p)\}^*u$ is 
compact and convex, there is a point $a^{(u)}$ 
in the boundary $\Bd(\{\partial F (p)\}^*u)$ 
such that 
$\|a^{(u)}\| 
= d_{T_pM} (o_p, \{\partial F (p)\}^*u) 
>0$ where $d_{T_pM}$ denotes the distance 
function of $T_pM$, i.e., $d_{T_xM} (a,b):= \|a-b\|$ for all $a,b \in T_xM$ ($x \in M$). 
Since $\Sph^{n-1}_{F(p)}$ is compact, there is 
a constant $\delta (p) > 0$ given by 
$\delta (p)
:= 
\min
\{ 
\|a^{(w)}\|\,|\, 
w \in \Sph^{n-1}_{F(p)} 
\}/ 2$, 
and hence 
$
d_{T_pM} (o_p, \{\partial F (p)\}^*u)
= \|a^{(u)}\| \ge 2\delta (p)$. 
By this inequality there is a constant 
$\ve(p) >0$ sufficiently small such that 
\begin{equation}\label{2018_01_11_lem2.12_1}
d_{T_pM} (o_p, \ol{\cU_{\ve (p)}(\{\partial F (p)\}^*)}u) \ge \delta (p)
\end{equation}
where 
$\ol{\cU_{\ve (p)}(\{\partial F (p)\}^*)}$ 
is the closure of the $\ve (p)$-open 
neighborhood $\cU_{\ve (p)}(\{\partial F (p)\}^*)$ 
of $\{\partial F (p)\}^*$ in $\cL (T_{F(p)}N, T_p M)$. 
Note that 
$\ol{\cU_{\ve (p)}(\{\partial F (p)\}^*)}u$ is a 
compact convex subset of $T_pM$. 
Indeed, let 
$\ol{\B_{\ve (p)}(\tilde{o})}$ be a closed ball 
with centre the origin $\tilde{o}$  
of $\cL (T_{F(p)}N, T_p M)$ and radius $\ve (p)$. 
Since $\ol{\B_{\ve (p)}(\tilde{o})}$ 
and $\{\partial F (p)\}^*$ are convex in 
$\cL (T_{F(p)}N, T_p M)$, 
and since 
$\ol{\cU_{\ve (p)}(\{\partial F (p)\}^*)} 
= \{\partial F (p)\}^* + \ol{\B_{\ve (p)}(\tilde{o})}$, 
\cite[Theorem 3.1]{Ro} shows 
that $\ol{\cU_{\ve (p)}(\{\partial F (p)\}^*)}$ is a 
compact convex subset in $\cL (T_{F(p)}N, T_p M)$, 
and hence $\ol{\cU_{\ve (p)}(\{\partial F (p)\}^*)} u$ 
is also in $T_pM$.\par 
Fix $x \in B_{r(p)}(p)$. 
Since $\|g^*(u)\|\ge \delta (p)$ 
for all $g^* \in \ol{\cU_{\ve (p)}(\{\partial F (p)\}^*)}$ 
by Eq.\,\eqref{2018_01_11_lem2.12_1}, and since 
$\tau_{x}^{p}$ is an isometry, 
$\|\tau_{x}^{p} (g^*(u))\| = \|g^*(u)\|\ge \delta (p)$ holds for all 
$g^* \in \ol{\cU_{\ve (p)}(\{\partial F (p)\}^*)}$, 
and hence 
$
d_{T_xM} 
( 
o_x, 
\tau_{x}^{p} 
( 
\ol{\cU_{\ve (p)}(\{\partial F (p)\}^*)}u
)
) \ge \delta (p)$. 
Since $\tau_{x}^{p}$ is linear, 
$\tau_{x}^{p} 
( 
\ol{\cU_{\ve (p)}(\{\partial F (p)\}^*)}u
)$ is a compact convex subset of $T_xM$. 
There is therefore a point 
$b^{(u,\,x)}\in 
\Bd
( 
\tau_{x}^{p} 
( 
\ol{\cU_{\ve (p)}(\{\partial F (p)\}^*)}u
)
)$ 
such that 
\begin{equation}\label{2018_02_10_lem2.12_b}
\|b^{(u,\,x)}\| 
= 
d_{T_xM} 
( 
o_x, 
\tau_{x}^{p} 
( 
\ol{\cU_{\ve (p)}(\{\partial F (p)\}^*)}u
)
) \ge \delta (p).
\end{equation}
Define a unit tangent vector 
$V_x^{(u)}$ at $x$ by 
$
V_x^{(u)}  := b^{(u,\,x)}/ \|b^{(u,\,x)}\| 
\in \Sph^{m-1}_x$. 
Lemma \ref{2018_02_13_lem3.10} shows that 
for $\ve (p)$ as above 
there is a constant $\lambda (p) \in (0, r(p)/2)$ 
for $\mu (p, \ve) = 2\lambda (p) := 2\lambda (p, \ve (p))$ 
such that 
$
\tau^{x}_{p} \circ \{\partial F (x)\}^* \circ \tau^{F(p)}_{F(x)} \subset 
\cU_{\ve (p)} (\{\partial F (p)\}^*)
$ 
holds for all $x \in B_{2\lambda (p)} (p)$. 
Lemma \ref{2018_02_01_lem2.2} 
gives assertions (i) and (ii). 
Since 
$\cU_{\ve (p)} (\{\partial F (p)\}^*) 
\subset 
\ol{\cU_{\ve (p)}(\{\partial F (p)\}^*)}$, 
for any $x \in B_{2\lambda (p)} (p)$ we obtain 
\begin{equation}\label{2018_02_10_lem2.12_d}
\{\partial F (x)\}^*\tau^{F(p)}_{F(x)}(u) 
\subset 
\tau_{x}^{p} 
( 
\ol{\cU_{\ve (p)}(\{\partial F (p)\}^*)}u
) \subset T_xM. 
\end{equation}

Define the line $\ell :\R \to T_xM$ 
by $\ell (t):= t\,V_x^{(u)}$. Note that 
$\ell$ is passing through 
$\tau_{x}^{p} 
( 
\ol{\cU_{\ve (p)}(\{\partial F (p)\}^*)}u
)$. Fix $g^* \in \{\partial F (x)\}^*$. 
From Eq.\,\eqref{2018_02_10_lem2.12_d} 
there is a unique constant $t_0 >0$ 
such that  
\begin{equation}\label{2018_02_10_lem2.12_e}
t_0 \ge \|b^{(u,\,x)}\| \quad 
\text{and} \quad 
\angle (
\ora{o_x\,\ell (t_0)}, \,
\ora{\ell (t_0)g^*(\tau^{F(p)}_{F(x)}(u))}
) = \frac{\pi}{2}, 
\end{equation}
i.e., $\ell (t_0)$ is the foot of the perpendicular 
from $g^*(\tau^{F(p)}_{F(x)}(u))$ to the line $\ell$. 
Set
$
\theta 
:= \angle 
(
\ora{o_x\,\ell (t_0)}, 
\ora{o_x\,g^*(\tau^{F(p)}_{F(x)}(u))}
)$. 
Note here that $\theta \in [0, \pi/2)$ because 
$\tau_{x}^{p} 
( 
\ol{\cU_{\ve (p)}(\{\partial F (p)\}^*)}u
)$ is convex in $T_xM$ 
and $o_x\not\in 
\tau_{x}^{p} 
( 
\ol{\cU_{\ve (p)}(\{\partial F (p)\}^*)}u
)$. 
It follows 
from 
Eqs.\,\eqref{2018_02_10_lem2.12_b} 
and \eqref{2018_02_10_lem2.12_e}  
that 
\begin{align}\label{2018_10_02_lem2.12_f}
\langle V_x^{(u)}, g^*(\tau^{F(p)}_{F(x)}(u))\rangle_M 
&
=\|g^*(\tau^{F(p)}_{F(x)}(u))\| \cos \theta
= \|\ell (t_0)\| = t_0 \ge \|b^{(u,\,x)}\|
\ge \delta (p),
\end{align}
which is assertion (iii).\par 
Furthermore, Eq.\,\eqref{2018_10_02_lem2.12_f} gives 
$
\delta (p) \le \langle 
V_x^{(u)}, (g^* \circ \tau^{F(p)}_{F(x)}) (u)
\rangle_M \le \| g^* (\tau^{F(p)}_{F(x)} (u)) \|$, 
which shows that every $g^* \in \{\partial F (x)\}^*$ is 
injective for all $x \in B_{2 \lambda (p)}(p)$. 
Since 
$\ra (g) = \ra (g^*)=n$ for all $g \in \partial F (x)$ as 
$x \in B_{2 \lambda (p)}(p)$, 
any point $x \in B_{2 \lambda (p)}(p)$ is nonsingular for $F$. 
Assertion (iv) thus holds.
$\qedd$
\end{proof}
 
\begin{proposition}\label{2018_01_28_prop2.12} 
Assume $m \le n$. 
If a point $p \in M$ is nonsingular for $F$, 
then there are two constants 
$\lambda (p) >0$ and 
$\delta (p)>0$ satisfying the following properties: 
\begin{enumerate}[{\rm (i)}] 
\item $B_{2 \lambda (p)} (p)$ satisfies \eqref{2018_01_28_Sect2_a} 
and \eqref{2018_01_28_Sect2_b} of Theorem \ref{Whitehead}; 
\item $F|_{B_{2 \lambda (p)} (p)}$ is 
a Lipschitz map from $B_{2 \lambda (p)} (p)$ 
into $B_{t (p)}(F(p))$; 
\item For any $u \in \Sph_p^{m-1}:=\{w \in T_p M\,|\,\|w\|=1\}$ and 
any $x \in B_{2 \lambda (p)} (p)$ there is 
a vector $V_{F(x)}^{(u)} \in \Sph_{F(x)}^{n-1}:=\{v \in T_{F(x)} N\,|\,\|v\|=1\}$ such that 
$
\langle (g \circ \tau^p_x) (u), V_{F(x)}^{(u)} \rangle_N \ge \delta (p)
$ 
holds for all $g \in \partial F(x)$; 
\item Every $x \in B_{2 \lambda (p)} (p)$ 
is nonsingular for $F$.
\end{enumerate}
\end{proposition}

\begin{proof}
Fix $p \in M$ nonsingular for $F$, and $u \in \Sph_p^{m-1}$. 
Set $\partial F(p)u:=\{g (u)\,|\, g \in \partial F(p)\}$. 
Remark \ref{2018_01_12_rem2.4} (ii) shows that 
$\partial F(p)u$ is compact and convex in $T_{F(p)}N$. 
Since $p$ is nonsingular for $F$, $\partial F(p)$ has maximal rank $m$, hence every $g\in \partial F(p)$ is injective, 
and $o_{F(p)} \not\in \partial F(p)u$. Thanks to Remark \ref{2018_01_12_rem2.4} (vi), 
the same argument as in the proof of Proposition \ref{2018_02_12_prop3.11} 
works for $\partial F(p)u$. Details are left to the reader.$\qedd$
\end{proof}

\begin{remark}We give here three remarks on Propositions 
\ref{2018_02_12_prop3.11} and \ref{2018_01_28_prop2.12}.  
\begin{enumerate}[{\rm (i)}]
\item Proposition \ref{2018_02_12_prop3.11}
is a completely new result.
\item 
Clarke first showed the same statement as in Proposition \ref{2018_01_28_prop2.12} 
in the case both $M$ and $N$ 
are Euclidean spaces 
of the same dimension, 
see \cite[Lemma 3]{C2}.
\item 
Without mentioning Proposition \ref{2018_01_28_prop2.12} we applied it in the proof 
of \cite[Lemma 2.21]{KT}. We did not 
give the proof there. It is provided here. 
\end{enumerate}
\end{remark}

\begin{corollary}\label{2018_09_20_cor3.14}
The set of all singular points of $F$ is a closed set in $M$.
\end{corollary}

\begin{proof}
Let $\Sing (F)$ be the set of all singular points of $F$. 
We will 
show that $M \setminus \Sing (F)$ is open in $M$. Fix $p \in M \setminus \Sing (F)$. 
We first assume $m \ge n$. By the property (iv) of Proposition \ref{2018_02_12_prop3.11} we can find a constant $\lambda (p) >0$ such that any point $x \in B_{2 \lambda (p)}(p)$ is nonsingular for $F$, hence 
$B_{2 \lambda (p)}(p) \subset M \setminus \Sing (F)$, 
and finally $M \setminus \Sing (F)$ is open.  
By applying that of Proposition \ref{2018_01_28_prop2.12}
the same proof works for $m \le n$.$\qedd$ 
\end{proof}

\section{Smooth approximation of Lipschitz 
maps between Riemannian manifolds}\label{2018_01_18_Sect4}

Working from results in \cite{GW}, \cite{G2}, \cite{GS}, \cite{H}, \cite{K}, and \cite{Sh1}, we define a smooth approximation 
of an arbitrary Lipschitz map between Riemannian manifolds. 
Throughout this section 
let $M$ be a compact Riemannian manifold 
of dimension $m$, $N$ a Riemannian 
manifold of dimension $n$, 
$d_M$ and $d_N$ the distance functions 
of $M$ and $N$, respectively, 
$F: M\to N$ a Lipschitz map, and $\inj (M)$ the injectivity radius of $M$. Note that $(0, \inj (M)/2) \subset \R$ is not empty because 
$M$ is compact. 

\begin{lemma}\label{2018_03_04_lem4.1}
There is a finite set $\{p_1, p_2, \ldots, p_k\} \subset M$ 
such that 
\begin{enumerate}[{\rm (I)}]
\item for each $p_i \in \{p_1, p_2, \ldots, p_k\}$ 
both $B_{r (p_i)}(p_i) \subset M$ and 
$B_{t (p_i)}(F(p_i)) \subset N$ satisfy the 
properties (i) and (ii) of Lemma \ref{2018_02_01_lem2.2} for $p=p_i$; 
\item $r (p_i) \in (0, \inj (M)/2)$ for all 
$p_i \in \{p_1, p_2, \ldots, p_k\}$; 
\item $M= \bigcup_{i=1}^k B_{r(p_i)}(p_i)$.
\end{enumerate} 
\end{lemma}

\begin{proof}
This follows immediately since $M$ is compact.$\qedd$ 
\end{proof}

By applying the Nash embedding theorem \cite{N}, 
$N$ can be isometrically embedded into the Euclidean space $\R^\ell$ 
with the canonical Riemannian metric 
$\langle\,\cdot\,,\,\cdot\,\rangle$ 
where $\ell \ge \max\{m, n+1\}$. 
$F$ can be regarded as a Lipschitz map 
from $M$ into $\R^\ell$, and hence we 
set 
\[
\wt{F}:= F : M \to N \subset \R^\ell.
\]
In the case where $N=\R$, this is not done.\par 
From now on we use the notation 
$\inj(M)$, 
$ \{B_{r(p_i)}(p_i)\}_{i=1}^{k}$, 
$\{B_{t(p_i)}(F(p_i))\}_{i=1}^{k}$, 
and $\wt{F}$ 
in the sense above.

\subsection{The local smooth approximation 
of Lipschitz maps}

In this subsection we define the local smooth approximation of 
$\wt{F}: M \to N \subset \R^\ell$ on each strongly convex ball $B_{r(p_i)} (p_i) \subset M$ with convolution smoothing.\par 
Fix $p_i \in \{p_1, p_2, \ldots, p_k\}$. 
Since 
$\exp_{p_i}|_{\B_{\inj(M)}(o_{p_i})}: 
\B_{\inj(M)}(o_{p_i}) \to B_{\inj(M)} (p_i)$ 
is a diffeomorphism, we can define the map 
$\cF^{(i)}: \B_{\inj(M)}(o_{p_i}) \to N \subset \R^\ell$ by 
\[
\cF^{(i)} := \wt{F} \circ \exp_{p_i}|_{\B_{\inj(M)}(o_{p_i})}.
\]
Choose an orthonormal basis 
$e_1^{(i)}, e_2^{(i)}, \ldots, e_m^{(i)}$ 
for $T_{p_i}M$. 
Using coordinates $(y_1^{(i)},y_2^{(i)}, \ldots, y_m^{(i)})$ with respect to $e_1^{(i)}, e_2^{(i)}, \ldots, e_m^{(i)}$ on $T_{p_i}M$, 
we identify $T_{p_i}M$ with $\R^m$. 
Let 
$(z_1,z_2, \ldots, z_\ell)$ be the standard coordinates 
of $\R^\ell$. 
We then have the coordinate representation 
$\cF^{(i)} = (\cF_1^{(i)}, \cF_2^{(i)}, \ldots, \cF_\ell^{(i)})$ 
of $\cF^{(i)}$ defined by 
$\cF_j^{(i)}
:= z_j \circ \cF^{(i)}$ 
for each $j\in \{1,2,\ldots, \ell\}$. 
Moreover let $\rho^{(i)} : T_{p_i}M \to \R$ be a smooth 
function given by 
\[
\rho^{(i)} (y) = 
\begin{cases}
\ \alpha \cdot e^{-1/(1-\|y\|^2)}  \ & (y \in \B_1 (o_{p_i})),\\
\ 0  \ & (y \in T_{p_i}M \setminus \B_1 (o_{p_i}))
\end{cases}
\]
where the constant $\alpha$ is chosen so that 
$\int_{y \,\in\, T_{p_i}M}\rho^{(i)} (y) dy =1$. 
For an $\ve \in (0, \inj(M)/2)$ the
Riemannian mollifier $\rho_\ve^{(i)}$ 
is then defined by 
$\rho_\ve^{(i)} (y) := \rho^{(i)}( y / \ve )/ \ve^m$ 
for all $y \in T_{p_i}M$, 
which is a nonnegative smooth function 
on $T_{p_i}M$ and satisfies  
\begin{equation}\label{2018_04_18_def4.1_2}
\supp \rho_\ve^{(i)} = \ol{\B_\ve (o_{p_i})} \quad 
\text{and} \quad 
\int _{T_{p_i}M} \rho_\ve^{(i)} (y)dy=1, 
\end{equation}
see for instance \cite{H}, \cite{Leoni}, 
or \cite{Y}. We now define the convolution smoothing 
of $\wt{F}$.

\begin{definition}\label{2018_01_20_def4.1} 
Fix $p_i \in \{p_1, p_2, \ldots, p_k\}$ and 
$\ve \in (0, \inj(M)/2)$. The map 
$\wt{F}_\ve^{(p_i)}:B_{r(p_i)}(p_i)\to \R^\ell$ 
is defined as follows. For any $q \in B_{r(p_i)}(p_i)$, 
\begin{align}\label{2018_01_20_def4.1_1}
&\wt{F}^{(p_i)}_\ve (q)
:=\int _{y\,\in\,T_{p_i}M} \rho_\ve^{(i)} (y) 
\cF^{(i)}(\exp^{-1}_{p_i} q-y) dy\\
& := \Big(
\int _{T_{p_i}M} \rho_\ve^{(i)} (y)\cF_1^{(i)}(\exp^{-1}_{p_i} q-y) dy,\, 
\ldots\,,
\int _{T_{p_i}M} \rho_\ve^{(i)} (y) 
\cF_\ell^{(i)}(\exp^{-1}_{p_i} q-y) dy \Big).\notag
\end{align}
\end{definition}

\begin{remark}\label{2018_01_25_rem4.2}
\begin{enumerate}[{\rm (i)}]
\item 
Since 
\[
\int _{T_{p_i}M} 
\rho_\ve^{(i)} (y)\cF_j^{(i)}(\exp^{-1}_{p_i} q-y) dy
= \int _{T_{p_i}M} 
\rho_\ve^{(i)} (\exp^{-1}_{p_i} q-y)
\cF_j^{(i)}(y) dy
\]
for each $j =1,2,\ldots,\ell$ (see for instance \cite{H}, 
\cite{Leoni}, or \cite{Y}), 
we have, for any $q \in B_{r(p_i)}(p_i)$, 
\[
\wt{F}^{(p_i)}_\ve (q)
=\int _{T_{p_i}M} \rho_\ve^{(i)} (y) 
\cF^{(i)}(\exp^{-1}_{p_i} q-y) dy
=\int _{T_{p_i}M} \rho_\ve^{(i)} (\exp^{-1}_{p_i} q-y) 
\cF^{(i)}(y) dy.
\]
\item
Fix $q \in B_{r (p_i)}(p_i)$ and 
$y \in \B_\ve (o_{p_i})$. We see, 
by Lemma \ref{2018_03_04_lem4.1} (II), that 
\begin{equation}\label{2018_01_25_rem4.2_2_b}
\|\exp^{-1}_{p_i} q-y\| \le \|\exp^{-1}_{p_i} q\| + \|y\| 
< r(p_i) + \ve < \inj (M),
\end{equation}
and hence $\cF^{(i)} (\exp^{-1}_{p_i} q-y)$ exists. 
Moreover, since $\supp \rho_\ve^{(i)} = \ol{\B_\ve (o_{p_i})}$,  
\begin{equation}\label{2018_01_25_rem4.2_2_a}
\wt{F}^{(p_i)}_\ve (q)
=\int _{y\,\in\,\B_\ve (o_{p_i})} \rho_\ve^{(i)} (y) 
\cF^{(i)}(\exp^{-1}_{p_i} q-y) dy
\end{equation}
holds, 
and hence $\wt{F}^{(p_i)}_\ve (q)$ exists.
\item 
Since each 
$
\int _{T_{p_i}M} 
\rho_\ve^{(i)} (y)\cF_j^{(i)}(\exp^{-1}_{p_i} q-y) 
dy
$ 
is smooth 
(see for instance \cite{H}, \cite{Leoni}, 
or \cite{Y}), 
$\wt{F}^{(p_i)}_\ve$ is smooth. 
\item 
In the case where $N=\R$, 
the 
convolution smoothing \eqref{2018_01_20_def4.1_1} of 
the Lipschitz function 
$F :M\to \R$ is given by 
\begin{equation}\label{2018_09_25_rem4.2_1} 
\wt{F}^{(p_i)}_\ve (q)
:=\int _{y\,\in\,T_{p_i}M} \rho_\ve^{(i)} (y) 
(F\circ \exp_{p_i})(\exp^{-1}_{p_i} q-y) dy
\end{equation}
for all $q \in B_{r(p_i)}(p_i)$. 
\end{enumerate}
\end{remark}

\begin{definition}
For each $\ve \in (0, \inj(M)/2)$ let 
\begin{equation}\label{2018_02_20_lem4.7_1}
\Lambda (\ve)
:=\max
\{
\L (\exp_{p_i}|_{\B_{r(p_i)+\ve}(o_{p_i})})\,|\, 
p_i \in \{p_1, p_2, \ldots, p_k\}
\}
\end{equation}
where $\L (\exp_{p_i}|_{\B_{r(p_i)+\ve}(o_{p_i})})$ 
is the Lipschitz constant of $\exp_{p_i}|_{\B_{r(p_i)+\ve}(o_{p_i})}$, i.e., 
\[
\L (\exp_{p_i}|_{\B_{r(p_i)+\ve}(o_{p_i})})
:=\sup
\left
\{\frac{d_M (\exp_{p_i} u, \exp_{p_i} v)}{\|u-v\|} \, \bigg| \,u,v \in \B_{r(p_i)+\ve} (o_{p_i} ), u\ne v
\right\}.
\] 
\end{definition}

\begin{remark}\label{2018_02_27_rem4.5}
Since $r(p_i)+\ve < \inj(M)$ for each 
$p_i \in \{p_1, p_2, \ldots, p_k\}$, 
$\exp_{p_i}|_{\B_{r(p_i)+\ve}(o_{p_i})}$ 
is a diffeomorphism, and hence 
$\Lambda (\ve)$ converges to a positive constant 
as $\ve \downarrow 0$.
\end{remark}

The next lemma tells us that $\wt{F}^{(p_i)}_\ve$ is a local 
smooth approximation of $\wt{F}$ on $B_{r (p_i)}(p_i)$.

\begin{lemma}{\rm (\cite[Lemma 2.16]{KT})}\label{2018_01_25_lem4.3} For each 
$\ve \in (0, \inj(M)/2)$ we have 
$\|\wt{F}^{(p_i)}_\ve (q)-\wt{F}(q)\|
\le \ve \cdot \Lambda (\ve) \cdot \L(F)$ 
for all $q\in B_{r (p_i)}(p_i)$ 
($i \in \{1, 2, \ldots, k\}$) 
where $\|\,\cdot\, \|$ is 
the Euclidean norm of $\R^\ell$, 
and $\L(F)$ denotes the Lipschitz constant 
of $F$, i.e., 
\[
\L(F):=\sup\left\{\frac{d_N(F(x),F(y))}{d_M(x,y)} \, \bigg| \,x,y\in M, x\ne y\right\}.
\]
\end{lemma}

Fix $p_i \in \{p_1, p_2, \ldots, p_k\}$ and 
$\ve \in (0, \inj(M)/2)$. We now construct 
Jacobi fields on $B_{\inj(M)} (p_i)$ from 
geodesic variations with the initial point $p_i$. 
For each $q \in B_{r(p_i)}(p_i)$ 
we set $\Sph_q^{m-1}:=\{v \in T_qM\,|\,\|v\|=1\}$. 
Fix $v \in \Sph_q^{m-1}$. 
For $\delta >0$ sufficiently small 
let $c_{v}: (-\delta, \delta)\to B_{r(p_i)}(p_i)$ 
be the minimal geodesic segment defined by 
$c_{v} (s):= \exp_q s v$. Since $c_{v} (s) \in B_{r(p_i)}(p_i)$ 
for all $s \in (-\delta, \delta)$, 
we observe, by the same argument 
as in Eq.\,\eqref{2018_01_25_rem4.2_2_b}, that 
\begin{equation}\label{2018_03_29_point}
\exp_{p_i}^{-1}c_{v} (s)-y \in \B_{\inj(M)} (o_{p_i})
\end{equation}
for all $s \in (-\delta, \delta)$ 
and $y \in \B_\ve (o_{p_i})$. 
Since
$\exp_{p_i} (\exp_{p_i}^{-1}c_{v} (s)-y) 
\in B_{\inj(M)} (p_i)$ 
holds for all $s \in (-\delta, \delta)$ 
and $y \in \B_\ve (o_{p_i})$ 
from Eq.\,\eqref{2018_03_29_point}, 
for each 
$y \in \B_\ve (o_{p_i})$ 
we can define the smooth map 
$\varphi_y^{(v)}: [0,1]\times (-\delta, \delta) \to B_{\inj (M)}(p_i)$ by 
$\varphi_y^{(v)} (t,s)
:= \exp_{p_i} t [\exp_{p_i}^{-1}c_{v} (s)-y]$. 
The map $\varphi_y^{(v)}$ is a geodesic 
variation with the initial point $p_i$ 
of the minimal geodesic segment 
\begin{equation}\label{2018_02_27_geodesic}
\varphi_y^{(v)} (t,0)
=\exp_{p_i} t (\exp_{p_i}^{-1}q-y)
\end{equation}
emanating from $\varphi_y^{(v)} (0,0)= p_i$ and 
ending at 
$\varphi_y^{(v)} (1,0)= \exp_{p_i}  (\exp_{p_i}^{-1}q-y)$. 
We get the Jacobi field 
\begin{equation}\label{2018_02_21_Jacobi}
J_y^{(v)} (t) := \frac{\partial \varphi_y^{(v)}}{\partial s}(t,0)
\end{equation}
along $\varphi_y^{(v)} (t,0)$, which satisfies the initial conditions 
$J_y^{(v)} (0)=o_{p_i}$ and $(DJ_y^{(v)} / dt) (0) 
= (d [\exp_{p_i}^{-1}c_{v} (s)-y] / ds) (0)$. 
For simplicity of notation we set  
\begin{equation}\label{2018_03_01_newpoint}
q_i (y):= \varphi_y^{(v)} (1,0)
= \exp_{p_i}  (\exp_{p_i}^{-1}q-y).
\end{equation}

\begin{remark}\label{2020_01_13_rem4.8}
The differential 
$(d\wt{F}_\ve^{(p_i)})_q (v)$ 
of $\wt{F}^{(p_i)}_\ve$ at $q$ is given 
for any 
$q \in B_{r(p_i)}(p_i)$ and $v \in \Sph_q^{m-1}$ by 
$
(d\wt{F}_\ve^{(p_i)})_q (v)
= \int_{y\,\in\,\B_\ve (o_{p_i})}
\rho_\ve^{(i)} (y)
d\wt{F}_{q_i (y)} 
(J_y^{(v)} (1))dy$. 
Indeed, fix $q \in B_{r(p_i)}(p_i)$ and $v \in \Sph_q^{m-1}$. Since $\wt{F}:M\to N\subset \R^\ell$, 
it follows from the definition of the differential of smooth maps (cf.\,\cite{Tu2011}),  
Eqs.\,\eqref{2018_01_25_rem4.2_2_a}, and \eqref{2018_03_01_newpoint} that  
\begin{align*}
(d\wt{F}_\ve^{(p_i)})_q (v)
&=\frac{d}{ds}\bigg|_{0} 
(\wt{F}_\ve^{(p_i)} \circ c_{v})(s)
= \frac{d}{ds}\bigg|_{0}
\int_{\B_\ve (o_{p_i})}
\rho_\ve^{(i)} (y)
\wt{F}(\varphi_y^{(v)} (1,s))dy\\
&= \int_{\B_\ve (o_{p_i})}\rho_\ve^{(i)}(y)
d\wt{F}_{q_i (y)}
\left(
\frac{\partial \varphi}{\partial s}(1,0)
\right)
dy 
= \int_{\B_\ve (o_{p_i})}\rho_\ve^{(i)}(y)
d\wt{F}_{q_i (y)} (J_y^{(v)} (1))dy.
\end{align*}
\end{remark}

\begin{lemma}\label{2018_02_20_lem4.7}
There is a constant $\omega (M) \in (0, \inj (M)/2)$ such that if $\ve \in (0, \omega (M))$, then parallel transport 
$\tau_{q_i (y)}^q: T_q M\to T_{q_i (y)} M$ 
is defined as in Definition \ref{2018_02_12_def2.1} 
for all $y \in \B_\ve (o_{p_i})$ and 
$q \in B_{r(p_i)}(p_i)$ ($i \in \{1, 2, \ldots, k\}$).
\end{lemma}

\begin{proof}We see, by  
Remark \ref{2018_02_27_rem4.5}, that 
for $\inj(M)/ 2$ there is a constant 
$\omega (M) \in (0, \inj (M)/2)$ such that 
$\ve \cdot \Lambda (\ve) < \inj(M) / 2$ 
for all $\ve \in (0, \omega (M))$. 
Fix $\ve \in (0, \omega (M))$, and let 
$y \in \B_\ve (o_{p_i})$ and $q \in B_{r(p_i)}(p_i)$. 
Since 
$\exp_{p_i}^{-1}q -y \in \B_{r(p_i)+\ve}(o_{p_i})$ 
by Eq.\,\eqref{2018_01_25_rem4.2_2_b}, 
and since 
$\exp_{p_i}^{-1}q \in \B_{r(p_i)+\ve}(o_{p_i})$, 
Eq.\,\eqref{2018_03_01_newpoint} gives 
$
d_M(q, q_i(y))
\le \Lambda (\ve)\cdot \ve 
< \inj(M)/2$, 
and hence there is a unique minimal geodesic segment 
emanating from $q$ to $q_i(y)$. The map  
$\tau_{q_i (y)}^q: T_q M\to T_{q_i (y)} M$ 
is therefore defined as claimed.$\qedd$
\end{proof}

\subsection{The global smooth approximation of 
Lipschitz maps}\label{subsect4.2_2021_03_25}

In this subsection 
we define the global smooth approximation 
$\wt{F}_\ve$ of $\wt{F}$ using  
local smooth approximations and a partition of 
unity argument.\par  
Since $M$ is compact, there is a smooth partition of 
unity $\{\psi_i\}_{i=1}^k$ subordinate to 
$\{B_{r(p_i)}(p_i)\}_{i=1}^k$ (cf.\,\cite{Tu2011}).

\begin{definition}\label{2018_02_25_def4.9}
Fix $\ve \in (0, \inj(M)/2)$. We define the smooth map 
$\wt{F}_\ve:M\to \R^\ell$ by 
\begin{equation}\label{2018_03_04_def4.9_1}
\wt{F}_\ve(q):=\sum_{i=1}^k \psi_i(q)\wt{F}^{(p_i)}_\ve(q) 
\qquad (q \in M) 
\end{equation}
where each $\wt{F}^{(p_i)}_\ve$ 
is the local smooth approximation of $\wt{F}$ 
on $B_{r(p_i)}(p_i)$. 
\end{definition}

The following lemma says that 
$\wt{F}_\ve$ is the global smooth 
approximation of $\wt{F}$. 

\begin{lemma}{\rm (\cite[Lemma 2.17]{KT})}\label{2018_02_25_lem4.10}
For each $\ve \in (0, \inj(M)/2)$, 
\[\|\wt{F}_\ve(q)-\wt{F}(q)\| 
\le \ve \cdot \Lambda (\ve) \cdot \L(F)
\] 
holds for all $q\in M$ 
where $\|\cdot\|$ denotes the Euclidean norm of $\R^\ell$ and 
$\Lambda (\ve)$ is the constant given by 
Eq.\,\eqref{2018_02_20_lem4.7_1}.
\end{lemma}


\section{Proof of Main Theorem (Theorem \ref{2018_01_04_maintheorem})}\label{2018_02_16_Sect5}

\subsection{Preliminaries}\label{2018_03_04_Sect5.1}

In this section it is shown that the smooth approximation of a Lipschitz map
defined in subsection \ref{subsect4.2_2021_03_25} is surjective near points that are nonsingular in the
sense of Clarke. 
Throughout this subsection 
let $M$ be a compact Riemannian manifold 
of dimension $m$, $N$ a Riemannian 
manifold of dimension $n$ with $m\ge n$, 
and $F: M\to N$ a Lipschitz map. 
Note here that we do not assume that $N$ 
is connected or compact.\par  
Via the Nash embedding theorem \cite{N}, 
we isometrically embed $N$ into Euclidean space $\R^\ell$ with the canonical Riemannian metric 
$\langle\,\cdot\,,\,\cdot\,\rangle$ where $\ell \ge \max 
\{m, n+1\}$. 
Let $\wt{F}:= F : M \to N \subset \R^\ell$, 
which is Lipschitz. Moreover we will use the same
notation as in Section \ref{2018_01_18_Sect4}, e.g., 
\begin{itemize}
\item $\inj(M)$ is the injectivity radius of $M$, 
\item $ \{B_{r(p_i)}(p_i)\}_{i=1}^{k}$, 
$ \{B_{t(p_i)}(\wt{F}(p_i))\}_{i=1}^{k}$ are families of a finite number $k$ of strongly convex balls 
$B_{r(p_i)}(p_i) \subset M$, $B_{t(p_i)}(\wt{F}(p_i)) \subset N$ satisfying (I)--(III) in Lemma \ref{2018_03_04_lem4.1},
\item $\{\psi_i\}_{i=1}^k$ is the smooth partition of 
unity subordinate to $\{B_{r(p_i)}(p_i)\}_{i=1}^k$, 
\item $\wt{F}^{(p_i)}_\ve: B_{r(p_i)}(p_i) \to 
\wt{F}^{(p_i)}_\ve(B_{r(p_i)}(p_i)) \subset \R^\ell$ is 
the local smooth approximation of $\wt{F}$, 
defined by Eq.\,\eqref{2018_01_20_def4.1_1}, 
\item $\wt{F}_\ve: M \to \wt{F}_\ve (M)\subset \R^\ell$ 
is the global smooth one of $\wt{F}$, done by Eq.\,\eqref{2018_03_04_def4.9_1}, etc.
\end{itemize}

In what follows let 
$p\in M$ be nonsingular for $\wt{F}$, and let 
$\lambda (p)$ be the positive constant as in Proposition \ref{2018_02_12_prop3.11}. Fix  
$q \in B_{\lambda (p)} (p)$. 
We can then choose $i \in \{1,2,\ldots, k\}$ satisfying 
$q \in \supp \psi_i$. 
Note that $\supp \psi_i \subset B_{r(p_i)}(p_i)$. 

\begin{lemma}\label{2018_03_04_lem5.1}
Set 
$
\ve^{(i)}(p):=\min 
\{ 
r (p_i), \, \omega (M), \, 
\lambda (p) / 
\L(
\exp_{p_i}|_{\B_{2r (p_i)}(o_{p_i})}
)
\}
$ 
where $\omega (M) \in (0, \inj (M)/2)$ denotes 
the constant as in Lemma 
\ref{2018_02_20_lem4.7}. 
Then for any $y \in \B_{\ve^{(i)}(p)} (o_{p_i})$ 
we have $q_i(y) \in B_{2\lambda (p)} (p)$ 
where each $q_i(y)$ is the point defined by 
Eq.\,\eqref{2018_03_01_newpoint}. 
In particular for any 
$y \in \B_{\ve^{(i)}(p)} (o_{p_i})$ 
parallel transport $\tau^{\wt{F}(p)}_{\wt{F}(q_i(y))}:T_{\wt{F}(p)}N\to T_{\wt{F}(q_i(y))}N$ along a unique minimal geodesic 
of $N$ emanating from $\wt{F}(p)$ 
to $\wt{F}(q_i(y))$ 
is defined in the sense 
of Definition \ref{2018_02_12_def2.1}.\par 
Note here that 
$\tau^{\wt{F}(p)}_{\wt{F}(q_i(y))}$ is not parallel translation 
along a line segment of $\R^\ell$ joining the two points. 
\end{lemma}

\begin{proof}
Fix $y \in \B_{\ve^{(i)}(p)} (o_{p_i})$. 
Since $\ve^{(i)}(p) \le r(p_i)$, 
the triangle inequality gives 
\[
\| \exp_{p_i}^{-1}q -y\| \le \|\exp_{p_i}^{-1}q\| +\|y\| 
< r(p_i) + \ve^{(i)}(p) \le 2 r(p_i), 
\]
and hence $\exp_{p_i}^{-1}q -y \in \B_{2 r(p_i)} (o_{p_i})$. 
Since $q \in B_{r(p_i)}(p_i)$, it is clear that 
$\exp_{p_i}^{-1}q \in \B_{2r(p_i)}(o_{p_i})$. 
Note that $\exp_{p_i}|_{\B_{2r (p_i)}(o_{p_i})}$ is a 
diffeomorphism, 
as $2r(p_i) < \inj(M)$, see Lemma \ref{2018_03_04_lem4.1} (II). 
We then see, by the triangle inequality, that  
$
d_M (p, q_i(y)) 
\le d_M (p,q)+ d_M(q,q_i(y)) < 
\lambda (p) 
+ 
\ve^{(i)}(p) \cdot \L(
\exp_{p_i}|_{\B_{2r (p_i)}(o_{p_i})})
\le 2\lambda (p)
$. 
Hence 
we get $q_i(y) \in B_{2\lambda (p)} (p)$ as claimed. 
Moreover, since $p, q_i(y) \in B_{2\lambda (p)} (p)$, it follows from Proposition \ref{2018_02_12_prop3.11} (ii) that 
$\wt{F}(p), \wt{F}(q_i(y)) \in B_{t(p)}(\wt{F}(p)) \subset N$. 
Along the minimal geodesic of $N$ emanating from 
$\wt{F}(p)$ to $\wt{F}(q_i(y))$ parallel transport 
$\tau^{\wt{F}(p)}_{\wt{F}(q_i(y))}:T_{\wt{F}(p)}N\to T_{\wt{F}(q_i(y))}N$ is 
defined as in Definition \ref{2018_02_12_def2.1}.
$\qedd$
\end{proof}

\begin{remark}
Since $\ve^{(i)}(p) \le \omega (M)$, 
and since $q \in B_{r(p_i)}(p_i)$, 
from Lemma \ref{2018_02_20_lem4.7} 
we have parallel transport 
$\tau^{q_i(y)}_q:T_{q_i(y)}M\to T_qM$, 
as in Definition \ref{2018_02_12_def2.1}, 
for all $y \in \B_{\ve^{(i)}(p)}(o_{p_i})$. 
We use this in the next lemma.  
\end{remark}

From now on $\delta (p) >0$ indicates the constant as in Proposition \ref{2018_02_12_prop3.11}, and for each $x \in M$ let 
$
\Sph_x^{m-1}
:=\{u \in T_x M \,|\, \|u\|=1\}
$ and 
$
\Sph_{\wt{F}(x)}^{n-1}
:=\{v \in T_{\wt{F}(x)} N \,|\, \|v\|=1\}
$.

\begin{lemma}
{\rm (Key Lemma)}\label{2018_03_04_lem5.3} 
Fix $\ve \in (0, \ve^{(i)}(p))$. 
For any $y \in \B_\ve (o_{p_i})$ 
and any 
$
\tilde{u} \in \Sph_{\wt{F}(q)}^{n-1}$ 
there is a vector 
$V_{q_i(y)}^{(\tilde{u})} \in \Sph_{q_i(y)}^{m-1}$ 
such that 
\begin{align*}
&\big\langle 
(d\wt{F}_\ve^{(p_i)})_q(\tau_q^{q_i(y)} (V_{q_i(y)}^{(\tilde{u})})),\tilde{u} 
\big\rangle\\
&\ge 
\delta(p)-\L(F)
\big(
\sup_{y\in \B_\ve( o_{p_i}) } 
\|J_y^{(\tau_q^{q_i(y)} (V_{q_i(y)}^{(\tilde{u})}))}(1)
-V_{q_i(y)}^{(\tilde{u})}\|
+
\sup_{y\in \B_\ve( o_{p_i})}\|\tilde{u}-
(\tau_{\wt{F}(q_i(y))}^{\wt{F}(p)} \circ \tau_{\wt{F}(p)}^{\wt{F}(q)})(\tilde{u})\|
\big).
\end{align*}
Here $J_y^{(\tau_q^{q_i(y)} (V_{q_i(y)}^{(\tilde{u})}))}$ 
is the Jacobi field, defined by 
Eq.\,\eqref{2018_02_21_Jacobi} for 
$v
= \tau_q^{q_i(y)} (V_{q_i(y)}^{(\tilde{u})}) \in \Sph_q^{m-1}$, 
along the geodesic 
$\varphi_y^{(\tau_q^{q_i(y)} (V_{q_i(y)}^{(\tilde{u})}))} (t,0)$ given by Eq.\,\eqref{2018_02_27_geodesic} joining $p_i$ 
to $q_i (y)$. 
\end{lemma}

\begin{proof}
By Lemma \ref{2018_03_04_lem5.1}, 
$q_i(y) \in B_{2 \lambda (p)}(p)$ holds for 
all $y \in \B_\ve (o_{p_i})$. Fix 
$\tilde{u} \in \Sph_{\wt{F}(q)}^{n-1}$. 
It follows from 
Proposition \ref{2018_02_12_prop3.11} (iii) 
for 
$
u
= \tau_{\wt{F}(p)}^{\wt{F}(q)}(\tilde{u}) 
\in \Sph_{\wt{F}(p)}^{n-1}
$ and 
$x= q_i(y)$ that  {\em for almost all} 
$y \in \B_\ve (o_{p_i})$ there is a vector 
$
V_{q_i(y)}^{(\tilde{u})}
:= 
V_{q_i(y)}^{(\tau_{\wt{F}(p)}^{\wt{F}(q)}(\tilde{u}))} 
\in \Sph_{q_i(y)}^{m-1}
$ 
such that 
\begin{equation}\label{2018_04_18_lem5.3_new2} 
\big\langle 
V_{q_i(y)}^{(\tilde{u})}, (d\wt{F}_{q_i(y)})^* 
(\tau^{\wt{F}(p)}_{F({q_i(y)})} 
(\tau_{\wt{F}(p)}^{\wt{F}(q)}(\tilde{u}))) 
\big\rangle_M 
\ge \delta (p)
\end{equation}
where $(d\wt{F}_{q_i(y)})^*$ is 
the adjoint of the differential 
$d\wt{F}_{q_i(y)} : 
T_{q_i(y)} M\to T_{\wt{F}(q_i(y))} N$, 
and $\langle \,\cdot\,, \,\cdot\,\rangle_M$ 
denotes Riemannian metric of $M$. 
Since $N$ is isometrically embedded into $\R^\ell$, 
we see, by the Riesz representation theorem (cf.\,\cite[Theorem 10.1]{Rom}) and 
Eq.\,\eqref{2018_04_18_lem5.3_new2}, that 
\begin{align}\label{2018_03_04_lem5.3_1}
\langle 
d\wt{F}_{q_i(y)} (V_{q_i(y)}^{(\tilde{u})}), 
(
\tau^{\wt{F}(p)}_{\wt{F}({q_i(y)})} 
\circ \tau_{\wt{F}(p)}^{\wt{F}(q)}
)
(\tilde{u})
\rangle
&=
\langle 
V_{q_i(y)}^{(\tilde{u})}, (d\wt{F}_{q_i(y)})^* 
(\tau^{\wt{F}(p)}_{F({q_i(y)})} 
(\tau_{\wt{F}(p)}^{\wt{F}(q)}(\tilde{u}))) 
\rangle_M\\
&\ge \delta (p)\notag
\end{align}
for almost all $y \in \B_\ve (o_{p_i})$. 
For simplicity of notation we set 
$v:= \tau_q^{q_i(y)} (V_{q_i(y)}^{(\tilde{u})}) \in \Sph_q^{m-1}$. 
We then see, 
by Eq.\,\eqref{2018_04_18_def4.1_2}, 
Remark \ref{2020_01_13_rem4.8}, 
and the Cauchy--Schwarz inequality, that 
\begin{align}\label{2018_03_04_lem5.3_2} 
&\langle 
(d\wt{F}_\ve^{(p_i)})_q (v), \tilde{u}
\rangle\\
&\ge - \L(F)\sup_{y \in \B_\ve (o_{p_i})}
\|
J_y^{(v)}(1)- V_{q_i(y)}^{(\tilde{u})}
\|
+\int_{\B_\ve (o_{p_i})} \rho_\ve^{(i)} (y) 
\langle 
d\wt{F}_{q_i(y)} (V_{q_i(y)}^{(\tilde{u})}), \tilde{u}
\rangle dy.\notag
\end{align}
Moreover, we see, 
by Eqs.\,\eqref{2018_04_18_def4.1_2},
\eqref{2018_03_04_lem5.3_1}, 
and the Cauchy--Schwarz inequality, 
that 
\begin{align}\label{2018_03_04_lem5.3_3} 
&\int_{\B_\ve (o_{p_i})} \rho_\ve^{(i)} (y) 
\big\langle 
d\wt{F}_{q_i(y)} (V_{q_i(y)}^{(\tilde{u})}), \tilde{u}
\big\rangle dy\\
&\ge
- \int_{\B_\ve (o_{p_i})} \rho_\ve^{(i)} (y) 
\|
d\wt{F}_{q_i(y)} 
(
V_{q_i(y)}^{(\tilde{u})}
)\| 
\cdot 
\| 
\tilde{u} 
- (
\tau^{\wt{F}(p)}_{\wt{F}({q_i(y)})} 
\circ \tau_{\wt{F}(p)}^{\wt{F}(q)}
)
(\tilde{u})
\| dy 
+ \delta (p)\notag\\
&\ge 
- \L(F)
\sup_{y \in \B_\ve (o_{p_i})}
\| 
\tilde{u} 
- (
\tau^{\wt{F}(p)}_{\wt{F}({q_i(y)})} 
\circ \tau_{\wt{F}(p)}^{\wt{F}(q)}
)
(\tilde{u})
\|
+ \delta (p).\notag
\end{align}
Substituting Eq.\,\eqref{2018_03_04_lem5.3_3} into Eq.\,\eqref{2018_03_04_lem5.3_2}, 
we obtain the desired inequality.
$\qedd$
\end{proof}

We can apply an argument almost identical to the proof of 
\cite[Lemmas 2.23, 2.24]{KT} 
and Lemma \ref{2018_03_04_lem5.3} 
to show the following. We omit the proof.

\begin{proposition}\label{2018_03_06_lem5.6}
There is a constant $\ve_0 (p)>0$ satisfying 
the following: For any 
$\tilde{u} \in \Sph_{\wt{F}(q)}^{n-1}$ there 
is a vector $w^{(\tilde{u})} \in\Sph_q^{m-1}$ 
such that for any $\ve\in(0,\ve_0(p))$, 
\begin{equation}\label{2018_03_08_prop5.6_1}
\big\langle
(d\wt{F}_\ve)_q(w^{(\tilde{u})}),\tilde{u}
\big\rangle \ge \frac{1}{3}\delta(p). 
\end{equation}
\end{proposition}

\begin{corollary}\label{2018_09_21_cor5.8}
If $N=\R$, and $p \in M$ is nonsingular for the 
Lipschitz function $F:M\to \R$, then there 
are two constants $\lambda (p) >0$ and $\ve_0 (p) >0$ 
such that 
$\grad \wt{F}_\ve \not= 0$ 
on $B_{\lambda (p)}(p)$ for all $\ve \in (0, \ve_0 (p))$, 
hence $\wt{F}_\ve$ has no critical points 
on $B_{\lambda (p)}(p)$. 
\end{corollary}

\begin{proof}
Let $\lambda (p)$ be the positive constant 
as in Proposition \ref{2018_02_12_prop3.11}. 
Fix $x \in B_{\lambda (p)}(p)$ 
and $u \in \Sph_{F(x)}^0$. 
By Proposition \ref{2018_03_06_lem5.6} there is a constant $\ve_0 (p)>0$ such that 
there is a vector $w^{(u)} \in\Sph_x^{m-1}$ 
satisfying 
$
\langle
(d\wt{F}_\ve)_x(w^{(u)}),u
\rangle \ge \delta(p)/ 3
$ 
for all $\ve\in(0,\ve_0(p))$. 
Fix $\ve \in (0, \ve_0 (p))$. 
Since 
$\tau^{F(x)}_{\wt{F}_\ve (x)} (u)
= (d/dt)|_{\wt{F}_\ve (x)}$, 
we have 
$
\delta(p)/ 3 
\le 
\langle
(d\wt{F}_\ve)_x(w^{(u)}), \tau^{F(x)}_{\wt{F}_\ve (x)} (u)
\rangle 
= 
\langle
w^{(u)} (\wt{F}_\ve )
(d/dt)|_{\wt{F}_\ve (x)}, 
(d/dt)|_{\wt{F}_\ve (x)} 
\rangle
=
w^{(u)}(\wt{F}_\ve) \cdot 1 
= 
\langle
(\grad \wt{F}_\ve)_x, w^{(u)} 
\rangle_M
\le 
\|
(\grad \wt{F}_\ve)_x
\|$, 
which shows the first assertion. 
The second assertion follows from the first one.
$\qedd$
\end{proof}

\subsection{Proof of Theorem \ref{2018_01_04_maintheorem}}\label{2018_03_11_Sect5.2}
We follow assumptions and notation 
of Section \ref{2018_03_04_Sect5.1}. 
In addition we assume that  
$N$ is connected and compact, and that 
the Lipschitz map $\wt{F}: M\to N \subset \R^\ell$ has no singular points on $M$.\par 
Since $N$ can be isometrically embedded into $\R^\ell$, 
it follows from the tubular neighborhood theorem 
(cf.\,\cite{H}, \cite{L}) via 
the normal exponential map $\exp^\perp:TN^\perp\to \R^\ell$ that there is a constant $\mu_0 >0$ such that 
$\exp^\perp$ is a diffeomorphism from 
an open neighborhood 
$\cU_{\mu_0} (O(TN^\perp)):=
\{
X\in TN^\perp\,|\,\|X\|<\mu_0\}$ of the zero section 
$O(TN^\perp) =\{o_x \in T_x N^\perp\,|\, x \in N
\}
$ 
onto an open one $\cU_{\mu_0}(N):= \exp^\perp [\cU_{\mu_0} (O(TN^\perp))]$ 
of $N$ in $\R^\ell$, 
which we will call the tubular neighborhood 
of $N$, where $o_x$ is the origin of $T_xN^\perp$. 
Since $\exp^\perp|_{\cU_{\mu_0} (O(TN^\perp))}$ 
is bijective, for any $y \in \cU_{\mu_0}(N)$ 
there is a unique point 
$(z, v) \in \cU_{\mu_0} (O(TN^\perp))$ 
such that $y= \exp^\perp (z,v)$. 
For such a pair $(y, (z,v))$ 
we have the smooth projection 
$\pi_N:\cU_{\mu_0}(N) \to N$ given by 
$\pi_N(y) = \pi_N(\exp^\perp (z,v)):= z$. 
Note that the first variation formula yields 
$\|y -\pi_N(y)\|= \inf_{x \in N} \|y-x\|$ 
for all $y \in \cU_{\mu_0}(N)$. 
For any $z \in N$ the definition of $\pi_N$ gives 
$(T_zN)^{\perp} = \Ker (d\pi_N)_z$.\par 
Since every $p \in M$ is nonsingular for $\wt{F}$, 
there are two positive constants 
$\delta (p)$ and $\ve_0(p)$ obtained 
in Propositions \ref{2018_02_12_prop3.11}
and \ref{2018_03_06_lem5.6}, 
which satisfy Eq.\,\eqref{2018_03_08_prop5.6_1}. 
Set 
$\delta_0:=\min\{\delta (p)\,|\, p \in M\}$ 
and $\ve_0:=\min\{\ve_0 (p)\,|\, p \in M\}$. 
Moreover Lemma \ref{2018_02_25_lem4.10} shows 
that for $\mu_0$ above there is 
a constant $\ve (\mu_0) \in (0, \inj(M)/2)$ 
such that if $\ve \in (0, \ve (\mu_0))$, then 
\begin{equation}\label{2018_12_13_subset}
\wt{F}_\ve (M) \subset \cU_{\mu_0} (N).
\end{equation}
Set 
$\ve_1:=\min \{\ve_0, \ve (\mu_0)\}$. 
It then follows from 
Proposition \ref{2018_03_06_lem5.6} for $q=p$ 
that 
for any $p \in M$ and any 
$\tilde{u} \in \Sph^{n-1}_{\wt{F}(p)}$ 
there is a vector $w^{(\tilde{u})} \in \Sph_p^{m-1}$ 
such that for any $\ve \in (0, \ve_1)$, 
\begin{equation}\label{2018_10_09_lem5.11_1}
\big\langle
(d\wt{F}_\ve)_p(w^{(\tilde{u})}), \tilde{u}\,
\big\rangle \ge \frac{1}{3}\delta_0.
\end{equation}

For any $x, y\in \R^\ell$
let 
$P_{y}^{x}: 
T_{x} \R^\ell 
\to 
T_{y} \R^\ell$ 
be the parallel translation along the line segment 
in $\R^\ell$ joining $x$ to $y$, 
and let 
$P^{y}_{x}
:= \big(P_{y}^{x}\big)^{-1}$.

\begin{lemma}\label{2018_10_09_lem5.11}
Fix $p\in M$ and $\ve \in (0, \ve_1)$. 
For any 
$
\wh{u} 
\in 
P_{\wt{F}_\ve (p)}^{\wt{F}(p)}
\big( 
\Sph_{\wt{F}(p)}^{n-1} 
\big)
$ 
there is a vector 
$
\wh{w} \in \Sph^{m-1}_p$ 
such that 
$
\angle 
(
(d\wt{F}_\ve)_p (\wh{w}), \wh{u}\,
)
<\pi/2
$ 
holds where 
$
\angle 
\big(
(d\wt{F}_\ve)_p (\wh{w}), \wh{u}\,
\big)
$ 
is the angle between 
$(d\wt{F}_\ve)_p (\wh{w})$ and $\wh{u}$ 
at the origin $o_{\wt{F}_\ve (p)}$ 
of $T_{\wt{F}_\ve (p)} \R^\ell$.
\end{lemma}

\begin{proof}
Fix $\wh{u} 
\in 
P_{\wt{F}_\ve (p)}^{\wt{F}(p)}
\big( 
\Sph_{\wt{F}(p)}^{n-1} 
\big)$. 
By Proposition \ref{2018_03_06_lem5.6},  
for $\tilde{u}:= 
P^{\wt{F}_\ve (p)}_{\wt{F}(p)}
\big( 
\wh{u}
\big) (\in \Sph^{n-1}_{\wt{F}(p)})$
there is a vector $\wh{w}:= w^{(\tilde{u})} \in \Sph_p^{m-1}$ 
with Eq.\,\eqref{2018_10_09_lem5.11_1}. 
Since 
$\langle
(d\wt{F}_\ve)_p(\wh{w}), \tilde{u}\,
\rangle
= 
\langle
(d\wt{F}_\ve)_p(\wh{w}), P_{\wt{F}_\ve (p)}^{\wt{F}(p)}
( 
\tilde{u}
)\,
\rangle$, we see 
$
0< \delta_0/ 3 
\le 
\langle
(d\wt{F}_\ve)_p(\wh{w}), \tilde{u}\,
\rangle
= 
\langle
(d\wt{F}_\ve)_p(\wh{w}), \wh{u}\,
\rangle 
=
\|
(d\wt{F}_\ve)_p(\wh{w})
\| 
\cos 
( 
\angle 
((
d\wt{F}_\ve)_p (\wh{w}), \wh{u}\,
))$, 
and finally 
$\angle 
\big(
(d\wt{F}_\ve)_p (\wh{w}), \wh{u}\,
\big)<\pi/2$. 
$\qedd$
\end{proof}

Since every $p \in M$ is nonsingular for $\wt{F}$, 
$\ra (g) = n$ holds 
for all $g \in \partial \wt{F} (p)$, and hence for each 
$\ve \in (0, \ve_1)$ 
we see, by Lemma \ref{2018_02_25_lem4.10} and Eq.\,\eqref{2018_10_09_lem5.11_1}, that 
\begin{equation}\label{2018_12_13_rank}
\Im (d\wt{F}_\ve)_p 
\cap 
P_{\wt{F}_\ve (p)}^{\wt{F}(p)} 
(\Ker (d\pi_N)_{\wt{F} (p)})
=\{o_{\wt{F}_\ve(p)}\} \qquad (p \in M).
\end{equation}
Moreover, by virtue of Eq.\,\eqref{2018_12_13_subset}, 
for each 
$\ve \in (0, \ve_1)$ 
we can define the smooth map 
$f_\ve :M\to N$ by 
\[
f_\ve (p) := (\pi_N \circ \wt{F}_\ve) (p), \qquad (p \in M).
\]
\begin{lemma}\label{2018_05_08_lem5.9}
For any $\eta >0$ there is a constant $\kappa(\eta) \in (0,\ve_1)$ such that if $\ve \in (0, \kappa(\eta))$, then
$d_N (f_\ve (p), \wt{F}(p))< \eta$ 
and 
$\Im (d\wt{F}_\ve)_p
\cap 
\Ker (d \pi_N)_{\wt{F}_\ve (p)} 
=\{ o_{\wt{F}_\ve (p)}\}$ 
hold for all $p \in M$.
\end{lemma}
 
\begin{proof}
Fix $p\in M$. 
By Lemma \ref{2018_02_25_lem4.10}, 
$
\lim_{\ve \downarrow 0}\|\wt{F}_\ve (p)-\wt{F}(p)\| =0
$, 
and since $\pi_N (\wt{F} (p))= \wt{F} (p)$, 
we have
$\lim_{\ve \downarrow 0}\|f_\ve (p)-\wt{F}(p)\|=0$.  
From this for any $\eta>0$ there is a constant 
$\alpha_1(p, \eta)\in (0,\ve_1)$ 
such that if $\ve \in (0, \alpha_1(p, \eta))$, then
\begin{equation}\label{2018_10_26_lem5.9_new2}
\|f_\ve (p)-\wt{F}(p)\|< \frac{\eta}{\eta +1}.
\end{equation}
Fix $\eta >0$. Since $N$ is isometrically embedded 
into $\R^\ell$, 
$\lim_{\ve \downarrow 0}\|f_\ve (p)-\wt{F}(p)\|=0$ 
also implies that there is a constant 
$\alpha_2(p, \eta)\in (0,\ve_1)$ 
such that if $\ve \in (0, \alpha_2(p, \eta))$, then
\begin{equation}\label{2018_10_26_lem5.9_new3}
\left|
\frac{d_N(f_\ve (p)-\wt{F}(p))}{\|f_\ve (p)-\wt{F}(p)\|}
- 1
\right|
< \eta. 
\end{equation}
Let 
$\beta_1(p,\eta):= \min\{\alpha_1(p, \eta), \alpha_2(p, \eta)\}$. 
Eqs.\,\eqref{2018_10_26_lem5.9_new2} and \eqref{2018_10_26_lem5.9_new3} show that if 
$\ve \in (0, \beta_1(p,\eta))$, then 
\begin{equation}\label{2018_10_26_lem5.9_new4}
d_N(f_\ve (p), \wt{F}(p))< \eta.
\end{equation}
For each $\ve \in (0, \beta_1(p,\eta))$ 
let $\gamma_\ve:[0, \mu_0)\to \cU_{\mu_0} (N)$ be 
a unit speed minimal geodesic 
emanating perpendicularly from $f_\ve (p)$ 
and passing through $\wt{F}_\ve (p)$. 
Eq.\,\eqref{2018_10_26_lem5.9_new4} shows 
that by letting $\ve \downarrow 0$, 
$\gamma_\ve$ converges to a unit speed minimal 
geodesic $\gamma_0 :[0, \mu_0)\to \cU_{\mu_0} (N)$ emanating perpendicularly from $\wt{F}(p)$. 
Since 
$\lim_{\ve \downarrow 0}\|\wt{F}_\ve (p)-\wt{F}(p)\| =0$, 
we see 
$\lim_{\ve \downarrow 0} \Ker (d\pi_N)_{\wt{F}_\ve (p)}
= \Ker (d\pi_N)_{\wt{F}(p)}$. 
Since 
$\Im (d\wt{F}_\ve)_p
\cap 
P_{\wt{F}_\ve (p)}^{\wt{F}(p)} 
(
\Ker (d\pi_N)_{\wt{F} (p)}
) 
=\{o_{\wt{F}_\ve(p)}\}$ 
by Eq.\,\eqref{2018_12_13_rank}, we see that 
for $\eta$ there is a constant $\beta_2(p,\eta) \in (0,\beta_1(p,\eta)$ such that if $\ve \in (0, \beta_2(p,\eta))$, then 
\begin{equation}\label{2018_05_08_lem5.9_8}
\Im (d\wt{F}_\ve)_p
\cap 
\Ker (d \pi_N)_{\wt{F}_\ve (p)} 
=\{ o_{\wt{F}_\ve (p)}\}. 
\end{equation}
By setting
$\kappa(\eta):=\min\{\beta_2(p,\eta)\,|\, p \in M\}$, 
Eqs.\,\eqref{2018_10_26_lem5.9_new4} 
and \eqref{2018_05_08_lem5.9_8} 
complete the proof.$\qedd$
\end{proof}

Fix $\ve \in (0, \kappa (\eta))$. 
Lemma \ref{2018_05_08_lem5.9} shows 
$\ra (d \pi_N|_{\Im (d\wt{F}_\ve)_p})=n$ 
for all $p\in M$, 
and hence $\ra ((df_\ve)_p)=n$ for all $p\in M$, which proves that $f_\ve$ is a 
smooth submersion from $M$ to $N$. Note that 
$f_\ve$ is an open map, because $f_\ve$ is locally equivalent to the canonical projection on some 
coordinate neighborhood of each point of $M$, 
see \cite{Tu2011}. Since $f_\ve$ is continuous, 
and since $M$ is compact, $f_\ve (M)$ is compact 
in $N$. $f_\ve (M)$ is thus closed in $N$, for $N$ is Hausdorff. 
Since $M$ is open in $M$, $f_\ve (M)$ is open in $N$. Connectedness of $N$ shows that 
$f_\ve$ is surjective. Let $K$ be 
any compact set in $N$. By virtue of the 
compactness of $N$, $K$ is closed in $N$. From 
the continuity of $f_\ve$ on $M$, $f_\ve^{-1} (K)$ 
is closed in $M$. Since $M$ is compact, 
$f_\ve^{-1} (K)$ is also, and hence $f_\ve$ is proper. 
Since $f_\ve$ is a proper and surjective submersion between compact, smooth manifolds, 
Ehresmann's lemma \cite{Ehr} shows that $f_\ve$ 
is a locally trivial fibration, i.e., an Ehresmann fibration.
$\qedd$

\section{Proof of Reeb's sphere theorem 
for Lipschitz functions (Theorem \ref{2019_06_08_thm1.7})}\label{2019_06_08_Sect6}

Throughout this section 
let $M$ be a closed Riemannian manifold of dimension 
$m$, and we assume that $M$ admits a Lipschitz 
function $F:M\to \R$ with exactly two singular 
points in the sense of Clarke, denoted by $z_1, z_2 \in M$.\par  
Since $z_1, z_2 \in M$ 
are singular for $F$, we see, 
by Lemma \ref{2018_07_08_lem2.21}, 
that 
$o_{z_i} \in \cdas_F(z_i)$ ($i=1,2$) 
where $\cdas_F(z_i)$ indicates the generalized 
gradient of $F$ at $z_i$ (see 
Definition \ref{2018_07_08_def2.18}). 
From the maximum and minimum values 
theorem 
we can therefore assume, without loss of generality, 
that $F(z_1)= \min_{x\,\in\,M}F(x)$ and 
$F(z_2)= \max_{y\,\in\,M}F(y)$. 
For simplicity of notation let 
$a_i:= F(z_i)$ for $i=1,2$. 
Note that $a_1< a_2$. 

\begin{lemma}\label{2019_06_08_lem6.1}
For any $r>0$ with $B_r(z_1)\cap B_r(z_2)= \emptyset$
there is a constant 
$b_i (r) \in (a_1,a_2)$ such that 
$F^{-1}(b_i (r)) \subset B_r(z_i)$ for each $i=1,2$.
\end{lemma}

\begin{proof}
We prove this lemma only in the case of $i=1$. 
Suppose not. There is then $r_0 > 0$ such that 
for any $\lambda \in (a_1,a_2)$, $F^{-1}(\lambda) \nsubseteq B_{r_0} (z_1)$ holds. 
For each $n \in \N$ there is 
$x_n \in F^{-1}( a_1 + (a_2-a_1) / 2n)$ 
such that $x_n \not\in B_{r_0} (z_1)$, 
and hence we get a sequence $\{x_n\}_{n \in \N}$ of such points $x_n$. Since $M$ is compact, 
$\{x_n\}_{n \in \N}$ has a convergent subsequence 
$\{x_{n_j}\}_{j \in \N}$. 
Let $\bar{x}:= \lim_{j\to\infty} x_{n_j}$. 
Since $F$ is continuous on $M$, 
we see that 
$
F(\bar{x}) 
= \lim_{j\to\infty} F(x_{n_j})
= \lim_{j\to\infty} 
\{ a_1 + (a_2-a_1) / 2n_j\}= a_1$. 
Now $\bar{x} \not= z_1$, 
and $F(\bar{x})= a_1$, hence $\bar{x}$ is a critical point of $F$, 
which is a contradiction.$\qedd$
\end{proof}

Fix $r >0$ with $B_r(z_1)\cap B_r(z_2)= \emptyset$. 
For $\ve \in (0, \inj(M)/2)$ 
let $\wt{F}_\ve: M \to \R$ 
be the global smooth approximation of $F$ 
defined by Eq.\,\eqref{2018_03_04_def4.9_1}.

\begin{lemma}\label{2019_06_08_lem6.4}
There is an open set $V$ of $M$ and 
a constant $\ve_0 \in (0, \inj(M)/2)$ 
such that if $\ve \in (0,\ve_0)$, then 
$F^{-1}([b_1 (r), b_2 (r)]) \subset V$, 
and $\wt{F}_\ve$ has no critical points on $V$.
\end{lemma}

\begin{proof}For simplicity of notation let 
$M':= F^{-1}([b_1 (r), b_2 (r)])$.
We first show $M' \subset V$. 
Since $z_1$ and $z_2$ are the only two critical points 
of $F$, and since $a_1<b_1(r) < b_2(r) <a_2$, 
$F$ has no critical points on $M'$. 
It follows from Lemma \ref{2018_07_08_lem2.21} and Corollary \ref{2018_09_21_cor5.8} that 
for each $p \in M'$ there are two  
constants $\lambda (p) >0$ and $\bar{\ve} (p) >0$ 
such that for each $\ve \in (0, \bar{\ve} (p))$, 
$\grad \wt{F}_\ve \not=0$ on $B_{\lambda (p)} (p)$. 
Since $M'$ is compact, 
there is a finite set 
$\{p_1, p_2, \ldots, p_k\} \subset M'$ such that 
$
M'
\subset 
\bigcup_{i=1}^k B_{\lambda (p_i)} (p_i)$. 
Since $\bigcup_{i=1}^k B_{\lambda (p_i)} (p_i)
$ is open in $M$, setting $V:= \bigcup_{i=1}^k B_{\lambda (p_i)} (p_i)$, we get the first assertion.\par 
We next show the second assertion. Set 
$\ve_0:= \min\{\bar{\ve} (p_1),\bar{\ve}(p_2),\ldots,\bar{\ve}(p_k)\}$. 
Since 
$
F^{-1}([b_1(r),b_2(r)]) \subset V = \bigcup_{i=1}^k B_{\lambda (p_i)} (p_i)
$ and $\grad \wt{F}_\ve \not=0$ on $V$, 
$\wt{F}_\ve$ has no critical points on $V$ 
for all $\ve \in (0, \ve_0)$.
$\qedd$
\end{proof}
 
\begin{lemma}\label{2019_06_09_lem6.6}
There is a constant $\ve_1 \in (0, \ve_0]$ 
such that if $\ve \in (0, \ve_1)$, then 
for any $c \in (b_1(r),b_2(r))$, 
$\wt{F}^{-1}_\ve ([b_1(r),b_2(r)])$ is diffeomorphic to 
$\wt{F}^{-1}_\ve (c) \times [b_1(r),b_2(r)]$.
\end{lemma}

\begin{proof}Fix $i \in \{1,2\}$. Since 
$B_r (z_i)\cap V$ is an open 
neighborhood of $F^{-1}(b_i (r))$, we see, by 
Lemma \ref{2018_02_25_lem4.10}, that 
there is a constant $\hat{\ve}_i \in (0, \inj(M)/2)$ 
such that if $\ve \in (0,\hat{\ve}_i)$, then 
$\wt{F}^{-1}_\ve (b_i (r)) 
\subset B_r (z_i)\cap V$. 
Let $\ve_1:=\min\{\ve_0, \hat{\ve}_1, \hat{\ve}_2\}$, 
and fix $\ve \in (0, \ve_1)$. 
Since $\wt{F}^{-1}_\ve (b_i (r)) 
\subset B_r (z_i)\cap V$, 
$\wt{F}^{-1}_\ve ([b_1(r),b_2(r)]) \subset V$ holds, 
and hence we see, by Lemma \ref{2019_06_08_lem6.4}, that $\wt{F}_\ve$ has no 
critical points on $\wt{F}^{-1}_\ve ([b_1(r),b_2(r)])$. 
From this, for each $t \in [b_1(r),b_2(r)]$, 
$\wt{F}^{-1}_\ve (t)$ is an $(m-1)$-dimensional 
compact regular submanifold of $M$. 
$\wt{F}^{-1}_\ve ([b_1(r),b_2(r)])$ is therefore diffeomorphic to 
$\wt{F}^{-1}_\ve (b_1) \times [b_1(r), b_2 (r)]$ 
by a well-known theorem in Morse theory (\cite[Theorem 2.31]{Matsumoto}, 
or \cite[Theorem 3.1]{Milnor2}). 
Fix $c \in (b_1(r),b_2(r))$. 
Since $\wt{F}^{-1}_\ve (s)$ and $\wt{F}^{-1}_\ve (t)$ are 
diffeomorphic for all $s, t \in [b_1(r),b_2(r)]$, 
$\wt{F}^{-1}_\ve (b_1)$ is diffeomorphic to 
$\wt{F}^{-1}_\ve (c)$, which yields our assertion.$\qedd$
\end{proof}

Now we give the proof of Theorem \ref{2019_06_08_thm1.7}. 
From Lemma \ref{2019_06_08_lem6.1} we observe 
$\lim_{r \downarrow 0} b_i (r)= a_i$ ($i=1,2$). 
Fix $c \in (b_1(r),b_2(r)) \subseteq (a_1, a_2)$. Let 
$r \downarrow 0$. Lemma \ref{2019_06_09_lem6.6} then shows that 
$M$ is homeomorphic to the suspension, denoted by 
$\Sigma$, of 
the compact regular submanifold $\wt{F}^{-1}_\ve (c)$ of $M$ 
for a sufficiently small $\ve >0$. It follows 
easily from a result of Brown \cite{Bro} 
(see also \cite[Introduction]{Edw}) that 
$\Sigma$ is homeomorphic to the $m$-sphere 
for a sufficiently small $\ve >0$, 
and $M$ is also.$\qedd$

\begin{flushleft}
K.~Kondo\\
Department of Mathematics, Faculty of Science, 
Okayama University\\
Okayama City, Okayama Pref. 700-8530, Japan\\
{\small e-mail: 
{\tt keikondo@math.okayama-u.ac.jp}}
\end{flushleft}


\addcontentsline{toc}{section}{References}
\begin{thebibliography}{M}

{\small
\bibitem{AbGr}
U.\,Abresch and D.\,Gromoll, 
{\em On complete manifolds with nonnegative Ricci curvature}, J. Amer. Math. Soc. \textbf{3} (1990), No. 2, 355--374. 

\bibitem{Ber}
M.\,Berger, {\em 
Les vari\'et\'es Riemanniennes (1/4)-pinc\'ees}, 
Ann. Scuola Norm. Sup. Pisa (3) \textbf{14} (1960), 
161--170.

\bibitem{Bro}
M.\,Brown, 
{\em The monotone union of open $n$-cells is an open 
$n$-cell}, Proc. Amer. Math. Soc. 
\textbf{12} (1961), 812--814.

\bibitem{Cerf}
J. Cerf, Sur les diff\'eomorphismes de la sph\`ere de dimension trois ($\Gamma_4=0$), Lecture Notes 
in Math., \textbf{53}. Springer-Verlag, 1968. 

\bibitem{Ch}
J.\,Cheeger, {\em Critical points of distance functions and applications to geometry}, in Geometric topology: recent developments (Montecatini Terme, 1990), 1--38, 
Lecture Notes in Math., \textbf{1504}. Springer, 
Berlin, 1991. 

\bibitem{CE}
J.\,Cheeger and D.G.\,Ebin, 
``Comparison Theorems in Riemannian Geometry", 
North-Holland Mathematical Library, {\bf 9}. 
North-Holland Publishing Co., Amsterdam-Oxford; American Elsevier Publishing Co., Inc., New York, 1975.

\bibitem{C1}
F.H.\,Clarke, 
{\em Generalized gradients and applications}, 
Trans. Amer. Math. Soc.
\textbf{205} (1975), 247--262.

\bibitem{C2}
F.H.\,Clarke, 
{\em On the inverse function theorem}, 
Pacific J. Math. \textbf{64} (1976), 97--102.

\bibitem{dC}
M.P. do Carmo, ``Riemannian Geometry", 
Birkh\"auser, Boston, MA, 1992. 
Translated from the second Portuguese edition 
by F.\,Flaherty. 

\bibitem{Edw}
R.D.\,Edwards, {\em Suspensions of homology spheres}, 
{\tt arXiv:0610573}

\bibitem{Ehr}
C.\,Ehresmann, 
{\em 
Les connexions infinit\'esimales dans un espace 
fibr\'e diff\'erentiable}, Colloque de topologie 
(espaces fibr\'es), Bruxelles, 1950, pp. 29--55. 
Georges Thone, Li\'ege; Masson et Cie., Paris, 1951. 

\bibitem{Fuji}
T.\,Fujioka, 
{\em A fibration theorem for collapsing sequences 
of  Alexandrov spaces}, to appear in J. Topol. Anal.; 
{\tt arXiv:1905.05484}

\bibitem{GH}
P.\,Goldstein and P.\,Haj\l asz, 
{\em Topological obstructions to continuity of Orlicz--Sobolev mappings of finite distortion}, 
Ann. Mat. Pura Appl. (4) \textbf{198} (2019), no.\,1, 243--262. 

\bibitem{GW}
R.E.\,Greene and H.\,Wu, 
{\em $C^\infty$ approximations of convex, subharmonic, and plurisubharmonic functions},  
Ann. Sci. \'Ecole Norm. Sup. (4) \textbf{12} (1979), 
no.\,1, 47--84. 

\bibitem{G1}
M.\,Gromov, 
{\em Curvature, diameter and Betti numbers}, Comment. Math. Helv. \textbf{56} (1981), 179--195.

\bibitem{G2}
M.\,Gromov, ``Partial Differential Relations", 
Ergeb. Math. Grenzgeb. (3), \textbf{9}. 
Springer--Verlag, Berlin, 1986.

\bibitem{Grove}
K. Grove, {\em Critical point theory for distance functions}, 
in Differential geometry: Riemannian geometry (Los Angeles, CA, 1990), 357--385, Proc. Sympos. Pure Math., {\bf 54}, 
Part 3. Amer. Math. Soc., Providence, RI, 1993. 

\bibitem{GP}
K.\,Grove and P.\,Petersen, 
{\em Bounding homotopy types by geometry}, 
Ann.\ of Math.\ (2) \textbf{128} (1988), 195--206.

\bibitem{GS}
K.\,Grove and K.\,Shiohama, {\em A generalized sphere theorem},
Ann.\ of Math.\ (2) \textbf{106} (1977), 201--211.

\bibitem{H}
M.W.\,Hirsch, ``Differential Topology'', 
Grad. Texts in Math., \textbf{33}. 
Springer--Verlag, New York, 1994. 
Corrected reprint of the 1976 original.

\bibitem{IT}
J.\,Itoh and M.\,Tanaka, 
{\em The Lipschitz continuity of the distance function to the cut locus}, 
Trans. Amer. Math. Soc. \textbf{353} (2001), no. 1, 21--40. 

\bibitem{K}
H.\,Karcher, {\em Riemannian center of mass 
and mollifier smoothing}, 
Comm. Pure Appl. Math. \textbf{30} (1977), 509--541. 

\bibitem{KervaireMilnor} 
M.\,Kervaire and J.\,Milnor, 
{\em Groups of homotopy spheres.\,I},  
Ann. of Math. (2) \textbf{77} (1963), 504--537.

\bibitem{KO}
K.\,Kondo and S.\,Ohta, 
{\em Topology of complete manifolds with radial curvature bounded from below}, Geom. Funct. Anal. \textbf{17} 
(4) (2007) 1237--1247.

\bibitem{KT2010}
K.\,Kondo and M.\,Tanaka, 
{\em Total curvatures of model surfaces control 
topology of complete open manifolds with radial 
curvature bounded below, II}, 
Trans. Amer. Math. Soc. \textbf{362} (2010),
6293--6324.

\bibitem{KT}
K.\,Kondo and M.\,Tanaka, {\em Approximations of Lipschitz maps via immersions and differentiable exotic sphere theorems}, Nonlinear Anal. \textbf{155} (2017), 219--249. 

\bibitem{L}
J.M.\,Lee, ``Introduction to Smooth Manifolds'' 
(Second edition), Grad. Texts in Math., \textbf{218}. 
Springer, New York, 2013. 

\bibitem{Leoni}
G.\,Leoni, ``A First Course in Sobolev Spaces'' 
(Second edition), Grad. Stud. Math., \textbf{181}. 
Amer. Math. Soc., Providence, RI, 2017. 

\bibitem{Li}
S.\,Li, {\em Cartan--Whitney presentation, non-smooth analysis and smoothability of manifolds: On a theorem of Kondo--Tanaka}, {\tt arXiv:1904.00515}

\bibitem{Matsumoto}
Y.\,Matsumoto, 
``An Introduction to Morse Theory", Translations of Mathematical Monographs, \textbf{208}. 
Amer. Math. Soc., Providence, RI, 2002.
Translated from the 1997 Japanese original by Kiki Hudson and Masahico Saito. 

\bibitem{Milnor2}
J.\,Milnor, 
``Morse theory'', Annals of Mathematics Studies, 
\textbf{51}. Princeton University Press, Princeton, N.J., 
1963. Notes by M.\,Spivak and R.\,Wells. 

\bibitem{Milnor3}
J.\,Milnor, ``Lectures on the $h$-cobordism Theorem", 
Princeton University Press, Princeton, N.J., 
1965. Notes by L.\,Siebenmann and J.\,Sondow.

\bibitem{N}
J.\,Nash, 
{\em The imbedding problem 
for Riemannian manifolds}, Ann.\ of Math. \textbf{63}, 
No.\,1 (1956), 
20--63.

\bibitem{Palais}
R.\,Palais, 
{\em Extending diffeomorphisms}, 
Proc. Amer. Math. Soc. \textbf{11} (1960), 274--277.

\bibitem{R}
H.\,Rademacher, {\em \"Uber partielle und totale differenzierbarkeit von Funktionen mehrerer Variabeln und \"uber die Transformation der Doppelintegrale}, 
Math. Ann. \textbf{79} (1919), 340--359.

\bibitem{Ree}
G.\,Reeb, {\em 
Sur certaines propri\'et\'es topologiques des 
vari\'et\'es feuillet\'ees}, Publ. Inst. Math. Univ. 
Strasbourg 11, pp. 5--89, 155--156. Actualit\'es 
Sci. Ind., no. 1183, Hermann $\&$ Cie., Paris, 1952.

\bibitem{Ro}
R.T.\,Rockafellar, ``Convex Analysis", Princeton Landmarks in Mathematics, 
Princeton University Press, Princeton, NJ, 1997. 
Reprint of the 1970 original. 

\bibitem{Rom}
S.\,Roman, ``Advanced Linear Algebra" 
(Third edition), Grad. Texts in Math. \textbf{135}, 
Springer-Verlag New York, 2008.

\bibitem{Sak}
T.\,Sakai, ``Riemannian Geometry", Translations of Mathematical Monographs \textbf{149}, 
Amer. Math. Soc., Providence, RI, 1996.
Translated from the 1992 Japanese original by the author. 

\bibitem{Sch}
R.\,Schneider, ``Convex Bodies: The Brunn--Minkowski Theory", Encyclopedia of Mathematics and its Applications, 
\textbf{44}. Cambridge University Press, Cambridge, 1993.

\bibitem{Sh1}
Y.\,Shikata, 
{\em On a distance function on the set of differentiable
structures}, Osaka J. Math. \textbf{3} (1966), 65--79. 

\bibitem{Sm0}
S.\,Smale, {\em Diffeomorphisms of the $2$-sphere}, 
Proc. Amer. Math. Soc. \textbf{10} (1959) 621--626.

\bibitem{Sm1}
S.\,Smale, {\em Generalized Poincar\'e's conjecture in dimensions greater than four}, Ann. of Math. \textbf{74} (1961), 391--406.

\bibitem{Sm2}
S.\,Smale, {\em On the structure of manifolds}, Amer. J. Math. \textbf{84} (1962), 387-399.

\bibitem{Toponogov}
V.A.\,Toponogov, 
{\em Riemann spaces with curvature bounded below} (Russian), Uspehi Mat. Nauk \textbf{14} (1959) no. 1 (85), 87--130. 

\bibitem{Tu2011}
L.W.\,Tu, ``An Introduction to Manifolds''  
(Second edition), Universitext, Springer, New York, 2011.

\bibitem{W}
A.D.\,Weinstein, {\em The cut locus and conjugate locus of a riemannian manifold}, 
Ann.\ of Math.\ (2) \textbf{87} (1968), 29--41.

\bibitem{WhZy}
R.L.\,Wheeden and A.\,Zygmund, 
``Measure and Integral", Marcel Decker, New York, 1977.

\bibitem{Wh}
J.H.C.\,Whitehead, {\em On the covering of a complete space by the geodesics through a point}, 
Ann.\ of Math. (2) \textbf{36} (1935), 679--704.

\bibitem{Yam1}
T.\,Yamaguchi, 
{\em Collapsing and pinching under a lower curvature bound}, Ann. of Math. (2) \textbf{133} (1991), no. 2, 
317--357. 

\bibitem{Yam2}
T.\,Yamaguchi, 
{\em A convergence theorem in the geometry of Alexandrov spaces}, Actes de la Table Ronde de G\'eom\'etrie Diff\'erentielle (Luminy, 1992), 601--642, 
S\'emin. Congr. \textbf{1}, Soc. Math. France, Paris, 1996. 

\bibitem{Y}
K.\,Yosida, ``Functional Analysis", 
Classics in Mathematics, Springer--Verlag, Berlin, 1995.
Reprint of the sixth (1980) edition. 
}
\end{thebibliography}
\end{document}